\pdfoutput=1
\RequirePackage{ifpdf}
\ifpdf % We~are running pdfTeX in pdf mode
\documentclass[pdftex]{sigma}
\else
\documentclass{sigma}
\fi

\usepackage{bm}
\usepackage{eucal}

\numberwithin{equation}{section}

\newtheorem{Theorem}{Theorem}[section]
\newtheorem{Lemma}[Theorem]{Lemma}
\newtheorem{Proposition}[Theorem]{Proposition}
 { \theoremstyle{definition}
\newtheorem{Definition}[Theorem]{Definition}
\newtheorem{Example}[Theorem]{Example}
\newtheorem{Remark}[Theorem]{Remark} }

\newcommand{\pr}[1]{#1^{\prime}}
\newcommand{\pf}{\operatorname{Pf}}
\newcommand{\tr}{{\operatorname{Tr}}}
\newcommand{\mfrak}[1]{\mathfrak{#1}}
\newcommand{\mcal}[1]{\mathcal{#1}}
\newcommand{\mbb}[1]{\mathbb{#1}}
\newcommand{\mrm}[1]{\mathrm{#1}}
\newcommand{\what}[1]{\widehat{#1}}

\begin{document}
%\allowdisplaybreaks

\newcommand{\arXivNumber}{2005.02837}

\renewcommand{\PaperNumber}{008}

\FirstPageHeading

\ShortArticleName{Pfaffian Point Processes from Free Fermion Algebras}

\ArticleName{Pfaffian Point Processes from Free Fermion Algebras:\\ Perfectness and Conditional Measures}

\Author{Shinji KOSHIDA}

\AuthorNameForHeading{S.~Koshida}

\Address{Department of Physics, Faculty of Science and Engineering, Chuo University,\\ Kasuga, Bunkyo, Tokyo 112-8551, Japan}
\Email{\href{mailto:koshida@phys.chuo-u.ac.jp}{koshida@phys.chuo-u.ac.jp}}

\ArticleDates{Received July 23, 2020, in final form January 16, 2021; Published online January 26, 2021}

\Abstract{The analogy between determinantal point processes (DPPs) and free fermionic calculi is well-known. We point out that, from the perspective of free fermionic algebras, Pfaffian point processes (PfPPs) naturally emerge, and show that a positive contraction acting on a ``doubled'' one-particle space with an additional structure defines a unique PfPP. Recently, Olshanski inverted the direction from free fermions to DPPs, proposed a scheme to construct a fermionic state from a quasi-invariant probability measure, and introduced the notion of perfectness of a probability measure. We propose a method to check the perfectness and show that Schur measures are perfect as long as they are quasi-invariant under the action of the symmetric group. We also study conditional measures for PfPPs associated with projection operators. Consequently, we show that the conditional measures are again PfPPs associated with projection operators onto subspaces explicitly described.}


\Classification{60G55; 46L53; 46L30}

\section{Introduction}
\subsection{Pfaffian point process}
In this paper, we assume that $\mfrak{X}$ is a countable set. The collection of point configurations in~$\mfrak{X}$ is identified with $\Omega=\Omega (\mfrak{X})=\{0,1\}^{\mfrak{X}}$, which is equipped with the product topology to be a~compact topological space. We regard each element $\omega\in \Omega$ as a function $\omega\colon \mfrak{X}\to\{0,1\}$ or a~collection of points $\omega=\{x_{i}\in\mfrak{X}\}_{i}$. We adopt the $\sigma$-algebra of Borel sets $\Sigma$. Then, for distinct points $x_{1},\dots, x_{n}\in \mfrak{X}$, the cylinder set
\begin{equation*}
	\Omega_{x_{1},\dots, x_{n}}=\{\omega\in\Omega\,|\,\omega (x_{1})=\cdots =\omega (x_{n})=1\}
\end{equation*}
is measurable.
Given a probability measure $M$ on $(\Omega,\Sigma)$, the $n$-point correlation function $\rho^{M}_{n}$, $n\in\mbb{N}$ is an $n$-variable symmetric function defined as the probability weight of cylinder sets:
\begin{equation*}
	\rho^{M}_{n}(x_{1},\dots, x_{n})=M(\Omega_{x_{1},\dots, x_{n}}),
\end{equation*}
where $x_{1},\dots, x_{n}\in \mfrak{X}$ are distinct. It is conventional to extend $\rho^{M}_{n}$ to a function on $\mfrak{X}^{n}$ so that it vanishes if any two points coincide.
Note that a system of correlation functions $\big\{\rho^{M}_{n}\big\}_{n\in\mbb{N}}$ determines the probability measure $M$ uniquely since the cylinder sets generate $\Sigma$.
A random variable $X$ with values in $(\Omega,\Sigma)$ is called a point process in $\mfrak{X}$. We also call a probability measure on $(\Omega,\Sigma)$ a point process not distinguishing a random variable from its distribution.

To define a Pfaffian point process, we need to fix some notations.
Suppose that a $(2\times 2)$-matrix-valued function $\mbb{K}(\cdot,\cdot)\colon \mfrak{X}\times\mfrak{X}\to M(2;\mbb{C})$ satisfying the anti-symmetry
\begin{equation}\label{eq:anti-sym_matrix_func}
	\mbb{K}(x,y)^{\mrm{T}}=-\mbb{K}(y,x),\qquad x, y\in\mfrak{X}
\end{equation}
is given. For any $n\in\mbb{N}$, we adopt the following identification of vector spaces:
\begin{align*}
	M(2;\mbb{C})\otimes M(n;\mbb{C}) & \xrightarrow{\sim} M(2n;\mbb{C}),\\
	 e_{i,j}\otimes e_{k,l}&\mapsto e_{2(k-1)+i, 2(l-1)+j},\qquad i,j=1,2,\quad k,l=1,\dots, n,
\end{align*}
where $e_{i,j}$ is the matrix with entry $1$ at the $(i,j)$-position and $0$ elsewhere.
Given points $x_{1},\dots, x_{n}\in\mfrak{X}$, we understand $[\mbb{K}(x_{i},x_{j})]_{1\le i,j\le n}$ as a $(2n\times 2n)$-matrix under the above identification and by regarding $x_{1},\dots, x_{n}$ as legs for $M(n;\mbb{C})$. The anti-symmetry (\ref{eq:anti-sym_matrix_func}) of the function $\mbb{K}(\cdot,\cdot)$ ensures the anti-symmetry of the $(2n\times 2n)$-matrix $[\mbb{K}(x_{i},x_{j})]_{1\le i,j\le n}$ so that
\begin{equation*}
[\mbb{K}(x_{i},x_{j})]_{1\le i,j\le n}^{\mrm{T}}=\big[\mbb{K}(x_{j},x_{i})^{\mrm{T}}\big]_{1\le i,j\le n}=-[\mbb{K}(x_{i},x_{j})]_{1\le i,j\le n}.
\end{equation*}
Therefore, the Pfaffian $\pf [\mbb{K}(x_{i},x_{j})]_{1\le i,j\le n}$ is defined.

\begin{Definition}
A probability measure $M$ on $(\Omega,\Sigma)$ is a Pfaffian point process (PfPP) if there exists a $(2\times 2)$-matrix-valued function $\mbb{K}^{M}(\cdot,\cdot)\colon \mfrak{X}\times\mfrak{X}\to M(2;\mbb{C})$ satisfying the anti-symmetry~(\ref{eq:anti-sym_matrix_func}) such that every $n$-point function with $n\in\mbb{N}$ admits a Pfaffian expression
\begin{equation*}
	\rho_{n}^{M}(x_{1},\dots, x_{n})=\pf \big[\mbb{K}^{M}(x_{i},x_{j})\big]_{1\le i,j\le n}.
\end{equation*}
We call the matrix-valued function $\mbb{K}^{M}(\cdot,\cdot)$ a correlation kernel of the PfPP $M$.
\end{Definition}

For a PfPP $M$, we write its correlation kernel as
\begin{equation*}
	\mbb{K}^{M}(x,y)
	=\left(
	\begin{matrix}
		\mbb{K}_{11}^{M}(x,y) & \mbb{K}_{12}^{M}(x,y) \\
		\mbb{K}_{21}^{M}(x,y) & \mbb{K}_{22}^{M}(x,y)
	\end{matrix}
	\right),\qquad x,y\in\mfrak{X}.
\end{equation*}
If, in particular, $\mbb{K}_{11}^{M}(x,y)=\mbb{K}_{22}^{M}(x,y)=0$ identically holds, we have
\begin{equation*}
	\mrm{Pf}\big[\mbb{K}^{M}(x_{i},x_{j})\big]_{1\le i,j\le n}=\det \big[\mbb{K}_{12}^{M}(x_{i},x_{j})\big]_{1\le i,j\le n}.
\end{equation*}
Therefore, we may naturally have the following definition as a special case of PfPPs.

\begin{Definition}
A PfPP $M$ on $(\Omega,\Sigma)$ is called a determinantal point process (DPP) if $\mbb{K}^{M}_{11}(x,y)\allowbreak =\mbb{K}^{M}_{22}(x,y)=0$ identically.
In this case, each correlation function admits a determinantal expression
\begin{equation*}
	\rho_{n}^{M}(x_{1},\dots, x_{n})=\det \big[K^{M}(x_{i},x_{j})\big]_{1\le i,j\le n},\qquad K^{M}(x,y)=\mbb{K}^{M}_{12}(x,y),\qquad x,y\in\mfrak{X}
\end{equation*}
We adopt the abuse of terminology to call the function $K^{M}(\cdot,\cdot)$ on $\mfrak{X}\times\mfrak{X}$ a correlation kernel of the DPP $M$.
\end{Definition}

DPPs form a significant class of point processes that have many applications to random partition and asymptotic representation theory \cite{Borodin1998a, BorodinOlshanski1998a, Borodin1998b,BorodinOlshanski1998b,Okounkov2001,Olshanski1998, Olshanski2003} (see also \cite{Borodin2011,BenHoughKrishnapurPeresVirag2006, Lyons2003, Soshnikov2000}). The analogy between DPPs and fermionic calculi is well-known \cite{Lytvynov2002,LytvynovMei2007,Olshanski2020, ShiraiTakahashi2003a, ShiraiTakahashi2003b, Spohn1987}. In particular, a quasi-free state of a certain class over an algebra of anti-commutation relations gives a DPP.
Compared to DPPs, there seem to be less attempts to unify an interplay between PfPPs and free fermions, while there are many preceding studies on concrete examples (incomplete references include \cite{BorodinRains2005, Ferrari2004, KatoriTanemura2007,Matsumoto2005, MatsumotoShirai2013,Nagao2007,Rains2000,Vuletic2007,WangLi2019}).
One of the aims of this paper is to extend a portion of the insight on DPPs found in the above mentioned literatures and to establish a framework to study PfPPs from free fermionic perspectives. Many settings, notations and problems are motivated by \cite{Olshanski2020}.

\subsection{From a positive contraction to a PfPP}
A well-known construction of a DPP on $\mfrak{X}$ \cite{Soshnikov2000} starts from an operator $K$ on $\ell^{2}(\mfrak{X})$ such that $K=K^{\ast}$ and $0\le K\le 1$, namely, a~positive contraction on~$\ell^{2}(\mfrak{X})$. Given a~positive contraction~$K$ on $\ell^{2}(\mfrak{X})$, there exists a DPP $M^{K}$ on $\mfrak{X}$ such that each correlation function is given by
\begin{equation*}
	\rho^{M^{K}}_{n}(x_{1},\dots, x_{n})=\det [K(x_{i},x_{j}) ]_{1\le i,j\le n},\quad K(x,y)=(e_{x},Ke_{y})_{\ell^{2}(\mfrak{X})}.
\end{equation*}
Here $e_{x}\in \ell^{2}(\mfrak{X})$, $x\in\mfrak{X}$ is a function defined by $e_{x}(y)=\delta_{x,y}$, $y\in\mfrak{X}$. Then the collection $\{e_{x}\}_{x\in\mfrak{X}}$ forms a complete orthonormal system of $\ell^{2}(\mfrak{X})$.

Here, we give a direct generalization of this result to PfPPs.
We define the complex conjugate~$J$ of $\ell^{2}(\mfrak{X})$ as an anti-linear operator on it that fixes each function~$e_{x}$, $x\in\mfrak{X}$.
We set $\mcal{K}=\ell^{2}(\mfrak{X})\oplus \ell^{2}(\mfrak{X})$ and take an anti-unitary involution $\Gamma$ on $\mcal{K}$ defined by
\begin{equation}
\label{eq:Gamma}
	\Gamma=\left(
	\begin{matrix}
	0 & J \\
	J & 0
	\end{matrix}
	\right)
\end{equation}
according to the prescribed direct sum decomposition.
For this pair $(\mcal{K},\Gamma)$, we consider the collection of operators
\begin{equation}
\label{eq:collection_covariance_ops}
	\mcal{Q}(\mcal{K},\Gamma)=\big\{S\in\mfrak{B}(\mcal{K})\,|\,0\le S=S^{\ast}\le 1,\, \overline{S}=1-S\big\},
\end{equation}
where $\mfrak{B}(\mcal{K})$ is the set of bounded operators on $\mcal{K}$, and for any operator $A\in\mfrak{B}(\mcal{K})$, we write $\overline{A}:=\Gamma A\Gamma$.
We will show that each operator $S\in\mcal{Q}(\mcal{K},\Gamma)$ uniquely determines a PfPP.

\begin{Proposition}\label{prop:existence_ppp}
Let us take $S\in\mcal{Q}(\mcal{K},\Gamma)$ and write it as
\begin{equation*}
	S=\left(
	\begin{matrix}
		S_{11} & S_{12} \\
		S_{21} & S_{22}
	\end{matrix}
	\right).
\end{equation*}
There exists a PfPP $M^{S}$ on $\mfrak{X}$ such that each correlation function is given by
\begin{equation*}
	\rho^{M^{S}}_{n}(x_{1},\dots, x_{n})=\mrm{Pf} [\mbb{K}_{S}(x_{i},x_{j}) ]_{1\le i,j\le n},
\end{equation*}
where
\begin{equation}
\label{eq:matrix_kernel_S}
	\mbb{K}_{S}(x,y)=\left(
	\begin{matrix}
		(e_{x},S_{21}e_{y})_{\ell^{2}(\mfrak{X})} & (e_{x},S_{22}e_{y})_{\ell^{2}(\mfrak{X})} \\
		(e_{x},(S_{11}-1)e_{y})_{\ell^{2}(\mfrak{X})} & (e_{x},S_{12}e_{y})_{\ell^{2}(\mfrak{X})}
	\end{matrix}
	\right),\qquad x,y\in\mfrak{X}.
\end{equation}
\end{Proposition}

\begin{Remark}
If $S$ is diagonal; $S_{12}=S_{21}=0$, the PfPP $M^{S}$ is a DPP.
\end{Remark}

\begin{Remark}
\label{rem:check_anti-sym_matrix_func}
For $S\in\mcal{Q}(\mcal{K},\Gamma)$, the function $\mbb{K}_{S}(\cdot,\cdot)$ satisfies the required anti-symmetry (\ref{eq:anti-sym_matrix_func}). In fact, the self-adjointness of $S$ gives $S_{11}^{\ast}=S_{11}$, $S_{22}^{\ast}=S_{22}$, $S_{12}^{\ast}=S_{21}$, and the identity $\overline{S}=1-S$ implies $JS_{11}J=1-S_{22}$, $JS_{12}J=-S_{21}$. In particular, we have $S_{12}^{\mrm{T}}=-S_{12}$, $S_{21}^{\mrm{T}}=-S_{21}$. Therefore,
\begin{gather*}
	(e_{x},S_{21}e_{y})_{\ell^{2}(\mfrak{X})} =-(e_{y},S_{21}e_{x})_{\ell^{2}(\mfrak{X})},\\
	(e_{x},S_{12}e_{y})_{\ell^{2}(\mfrak{X})} =-(e_{y},S_{12}e_{x})_{\ell^{2}(\mfrak{X})}, \\
	(e_{x},S_{22}e_{y})_{\ell^{2}(\mfrak{X})} =\overline{(e_{y},J(1-S_{11})Je_{x})_{\ell^{2}(\mfrak{X})}}=-(e_{y},(S_{11}-1)e_{x})_{\ell^{2}(\mfrak{X})}, \\
(e_{x},(S_{11}-1)e_{y})_{\ell^{2}(\mfrak{X})}=-\overline{(e_{y},JS_{22}Je_{x})_{\ell^{2}(\mfrak{X})}}=-(e_{y},S_{22}e_{x})_{\ell^{2}(\mfrak{X})}
\end{gather*}
for any $x,y\in\mfrak{X}$, which implies the anti-symmetry (\ref{eq:anti-sym_matrix_func}).
\end{Remark}

It is standard to restate Proposition~\ref{prop:existence_ppp} in terms of a Fredholm Pfaffian. Let us assume that a matrix-valued function $\mbb{K}\colon \mfrak{X}\times\mfrak{X}\to M(2;\mbb{C})$ is finitely supported. We take another matrix-valued function $\mbb{J}\colon \mfrak{X}\times\mfrak{X}\to M(2;\mbb{C})$ defined by
\begin{equation*}
	\mbb{J}(x,y)=\delta_{x,y}\left(
	\begin{matrix}
	0 & 1 \\
	-1 & 0
	\end{matrix}
	\right),\qquad x,y\in\mfrak{X}.
\end{equation*}
Then, the sum $\mbb{J}+\mbb{K}$ still exhibits the anti-symmetry~(\ref{eq:anti-sym_matrix_func}).
For each $Y=\{x_{1},\dots,x_{n}\}\subset\mfrak{X}$, it can be verified that~\cite{Rains2000}
\begin{equation*}
	\pf [(\mbb{J}+\mbb{K})(x_{i},x_{j}) ]_{1\le i,j\le n}=1+\sum_{X\subset Y}\pf [\mbb{K}(x,y) ]_{x,y\in X},
\end{equation*}
where the sum runs over non-empty subsets $X\subset Y$. Since $\mbb{K}$ is now supposed to be finitely supported, this description gets stable under the limit $Y \to \mfrak{X}$ so that the following definition of the Fredholm Pfaffian makes sense:
\begin{equation*}
	\pf [\mbb{J}+\mbb{K} ]_{\mfrak{X}}:=1+\sum_{X\subset\mfrak{X}}\pf [\mbb{K}(x,y) ]_{x,y\in X},
\end{equation*}
where the sum over $X$ in fact reduces to a finite sum. It is, of course, possible to extend the definition of Fredholm Pfaffian to a not-necessarily finitely supported function $\mbb{K}$, but we will not need such a generality.

Let $\alpha\colon \mfrak{X}\to\mbb{R}$ be a function such that $\alpha (x)\ge 1$, $x\in\mfrak{X}$ and $\alpha-1$ is finitely supported. Given an anti-symmetric matrix-valued function $\mbb{K}\colon \mfrak{X}\times\mfrak{X}\to M(2;\mbb{C})$, we understand a new function denoted as $\sqrt{\alpha-1}\mbb{K}\sqrt{\alpha-1}$ as
\begin{equation*}
	\left(\sqrt{\alpha-1}\mbb{K}\sqrt{\alpha-1}\right)(x,y):=\sqrt{\alpha (x)-1}\mbb{K}(x,y)\sqrt{\alpha (y)-1},\qquad x,y\in\mfrak{X}.
\end{equation*}
Obviously, the new function $\sqrt{\alpha-1}\mbb{K}\sqrt{\alpha-1}$ again exhibits the anti-symmetry~(\ref{eq:anti-sym_matrix_func}) and is finitely supported.
To such a function $\alpha$, we can associate a multiplicative functional $\Psi_{\alpha}$, which is a measurable function on $(\Omega,\Sigma)$ defined by
\begin{equation*}
	\Psi_{\alpha}(\omega)=\prod_{x\in\omega}\alpha (x),\qquad \omega\in \Omega.
\end{equation*}
Notice that the infinite product reduces to a finite one since $\alpha -1$ is finitely supported.
In terms of these notions, Proposition~\ref{prop:existence_ppp} is equivalent to the following one:
\begin{Proposition}\label{prop:existence_ppp_fredholm}
Let $S\in\mcal{Q}(\mcal{K},\Gamma)$. There is a unique PfPP $M^{S}$ on $\mfrak{X}$ possessing the following property: for any function $\alpha$ on $\mfrak{X}$ such that $\alpha (x)\ge 1$, $x\in\mfrak{X}$ and $\alpha -1$ is finitely supported,
\begin{equation*}
	\int_{\Omega}\Psi_{\alpha}(\omega) M^{S}({\rm d}\omega)=\pf\big[\mbb{J}+\sqrt{\alpha-1}\mbb{K}_{S}\sqrt{\alpha-1}\big]_{\mfrak{X}}.
\end{equation*}
\end{Proposition}
The equivalence between Propositions~\ref{prop:existence_ppp} and~\ref{prop:existence_ppp_fredholm} basically follows from~\cite{Rains2000}. We will present a proof of this equivalence in Section~\ref{sect:existence_PfPP} for readers' convenience.

\begin{Remark} In \cite{Matsumoto2005b}, the author considered a problem concerning the existence of an $\alpha$-deformed PfPP for a given correlation kernel. Note that the case of $\alpha=-1$ is that of \mbox{PfPPs}. To compare the result at $\alpha=-1$ therein to ours, we regard a matrix-valued function $\mfrak{X}\times\mfrak{X}\to M(2;\mbb{C})$ as an operator on $\mcal{K}$. Notice, in particular, that the function $\mbb{J}$ is identified with
\begin{equation*}
	\mbb{J}=\left(
	\begin{matrix}
		0 & I \\
		-I & 0
	\end{matrix}
	\right)
\end{equation*}
according to the direct sum decomposition $\mcal{K}=\ell^{2}(\mfrak{X})\oplus\ell^{2}(\mfrak{X})$. The case considered in \cite{Matsumoto2005b} was that, for a correlation kernel $\mbb{K}$, the product $\mbb{J}\mbb{K}$ is self-adjoint and the author gave sufficient conditions so that there is a corresponding PfPP. On the other hand, in our case, a correlation kernel $\mbb{K}_{S}$ is related to $S\in\mcal{Q}(\mcal{K},\Gamma)$ as
\begin{equation*}
	\mbb{V}\mbb{K}_{S}=S-
	\left(
	\begin{matrix}
	I & 0 \\
	0 & 0
	\end{matrix}
	\right),\qquad \mbox{where}\quad
	\mbb{V}=\left(
	\begin{matrix}
	0 & I \\
	I & 0
	\end{matrix}
	\right).
\end{equation*}
Hence, in particular, $\mbb{V}\mbb{K}_{S}$ is self-adjoint. When we further impose that $\mbb{J}\mbb{K}_{S}$ is self-adjoint, we must have $S_{12}=S_{21}=0$.
Therefore, the intersection between the result in~\cite{Matsumoto2005b} at $\alpha=-1$ and ours is the case of DPPs.
PfPPs of the same symmetry class as in~\cite{Matsumoto2005b} have also been studied in~\cite{Kargin2014,KasselLevy2019} using quaternion determinants.
\end{Remark}

An interesting subclass of PfPPs obtained in this manner consists of those associated with projection operators.
We write the collection of projection operators in $\mcal{Q}(\mcal{K},\Gamma)$ as
\begin{equation*}
	\mrm{Gr}(\mcal{K},\Gamma)=\big\{P\in\mcal{Q}(\mcal{K},\Gamma)\,|\,P^{2}=P\big\}.
\end{equation*}
This notation is, of course, motivated by the fact that a projection operator $P\in\mrm{Gr}(\mcal{K},\Gamma)$ determines a closed subspace $P\mcal{K}\subset\mcal{K}$ and, therefore, the collection of projection operators can be regarded as an analogue of the Grassmann variety.
Let $P_{0}\in \mrm{Gr}(\mcal{K},\Gamma)$ be the projection operator onto the first component of the direct sum decomposition $\mcal{K}=\ell^{2}(\mfrak{X})\oplus\ell^{2}(\mfrak{X})$, which is expressed as
\begin{equation}
\label{eq:projection_vacuum}
	P_{0}=\left(
	\begin{matrix}
		I & 0 \\
		0 & 0
	\end{matrix}
	\right).
\end{equation}
For each $n\in\mbb{Z}_{\ge 0}$, we write $\Omega_{n}=\{\omega\in \Omega\,|\, \#\omega=n\}$ for the collection of $n$-point configurations and set $\Omega^{\circ}=\bigcup_{n=0}^{\infty}\Omega_{n}$, which consists of configurations of finitely many points.

\begin{Proposition}\label{prop:finitary_PfPP}
Let $P\in \mrm{Gr}(\mcal{K},\Gamma)$ be a projection operator such that $P-P_{0}$ is of Hilbert--Schmidt class. Then the associated PfPP $M^{P}$ is supported in $\Omega^{\circ}$. Equivalently, a~point process~$X$ in~$\mfrak{X}$ obeying $M^{P}$ satisfies $\# X<\infty$ almost surely.
\end{Proposition}

\subsection{CAR algebra and quasifree states}
We introduce an algebra of canonical anti-commutation relations (CAR algebra, for short) with a general one-particle Hilbert space following \cite{Araki1970,Binnenhei1995}.
Another style of, but equivalent, definition of a CAR algebra, can be found in, e.g., \cite{BratteliRobinson1997} (see Remark~\ref{rem:equivalence_CAR} below for the equivalence).
Let~$\mcal{K}$ be a complex Hilbert space of infinite dimension and $\Gamma$ be an anti-unitary involution on $\mcal{K}$.
The algebra $\mcal{C}_{0}(\mcal{K},\Gamma)$ is a~$\ast$-algebra over $\mbb{C}$ generated by $B(f)$, $f\in\mcal{K}$ subject to the relations
\begin{alignat*}{3}
	&B(\alpha f+\beta g)=\alpha B(f)+\beta B(g),\qquad && f, g\in\mcal{K},\quad \alpha,\beta\in\mbb{C},& \\
	&B(f)^{\ast}= B(\Gamma f),\qquad && f\in\mcal{K},&\\
	&\{B(f)^{\ast}, B(g)\}=(f,g)_{\mcal{K}}, \qquad && f,g\in\mcal{K},&
\end{alignat*}
where $\{\cdot,\cdot\}$ is the anti-commutator; $\{a,b\}:=ab+ba$.
It is known that the algebra $\mcal{C}_{0}(\mcal{K},\Gamma)$ admits a unique C$^{\ast}$-norm $\|\cdot\|$. We denote the C$^{\ast}$-completion by $\mcal{C}(\mcal{K},\Gamma)=\overline{\mcal{C}_{0}(\mcal{K},\Gamma)}^{\|\cdot\|}$ and call it the (self-dual) CAR algebra with one-particle space $(\mcal{K},\Gamma)$. To those who are more familiar with regarding $\ell^{2}(\mfrak{X})$ as a one-particle Hilbert space, we emphasize that we adopt a~``doubled'' space as a one-particle Hilbert space.

For a general C$^{\ast}$-algebra $\mfrak{A}$, a state over it is, by definition, a linear functional $\varphi\colon \mfrak{A}\to\mbb{C}$ satisfying the conditions that
\begin{enumerate}\itemsep=0pt
\item[1)] 	for every $A\in\mfrak{A}$, $\varphi (A^{\ast}A)\ge 0$ holds,
\item[2)] 	it is normalized:
		\begin{equation*}
			\|\varphi\|:=\sup\big\{|\varphi (A)|\, \big|\, \|A\|=1\big\}=1.
		\end{equation*}
\end{enumerate}
Note that, from these properties, it can be deduced that $\varphi (1)=1$ (see, e.g., \cite[Chapter~I, Section~9]{Takesaki1979}). Since any positive element $B\in\mfrak{A}$ admits an expression $B=A^{\ast}A$ with some $A\in\mfrak{A}$, the first condition is equivalently stated that $\varphi (B)\ge 0$ for all positive elements $B\in \mfrak{A}$.

\begin{Definition}A state $\varphi$ over a CAR algebra $\mcal{C}(\mcal{K},\Gamma)$ is said to be quasi-free if the moments admit Wick's formula, i.e.,
\begin{gather*}%\label{eq:Wick-property1}
	\varphi (B(f_{1})\cdots B(f_{2n+1}) ) =0, \\
%\label{eq:Wick-property2}
	\varphi (B(f_{1})\cdots B(f_{2n}) ) =(-1)^{n(n-1)/2}\sum_{\sigma}\operatorname{sgn} \sigma \prod_{i=1}^{n}\varphi (B(f_{\sigma (i)})B(f_{\sigma (i+n)}) ),
\end{gather*}
where $\sigma$ runs over permutations in $\mfrak{S}_{2n}$ such that $\sigma(1)<\cdots<\sigma (n)$ and $\sigma (i)<\sigma (i+n)$, $i=1,\dots, n$.
\end{Definition}

It is immediate from the definition that, for a quasi-free state $\varphi$ over $\mcal{C}(\mcal{K},\Gamma)$, a $2n$-point correlation function is expressed in terms of a Pfaffian so that
\begin{equation*}
	\varphi (B(f_{1})\cdots B(f_{2n}) )=\pf \mbb{A}_{\varphi}(f_{1},\dots, f_{2n}),
\end{equation*}
where $\mbb{A}_{\varphi}(f_{1},\dots, f_{2n})$ is the unique anti-symmetric $(2n\times 2n)$-matrix defined by
\begin{equation*}
	\mbb{A}_{\varphi}(f_{1},\dots, f_{2n})_{i,j}=\varphi (B(f_{i})B(f_{j}) ),\qquad 1\le i <j \le 2n.
\end{equation*}

By definition, a quasi-free state over a CAR algebra is uniquely determined by the two-point function. In fact, we have the following.
\begin{Lemma}[\cite{Araki1970}]\label{lem:quasi-free_covariance}
The collection of quasi-free states over $\mcal{C}(\mcal{K},\Gamma)$ is in one-to-one correspondence with the collection of operators $\mcal{Q}(\mcal{K},\Gamma)$ defined in~\eqref{eq:collection_covariance_ops}, under which an operator $S\in\mcal{Q}(\mcal{K},\Gamma)$ corresponds to the quasi-free state~$\varphi_{S}$ defined by
\begin{equation*}%\label{eq:quasifree_covariance}
	\varphi_{S}\big(B(f)^{\ast}B(g)\big)=(f,Sg)_{\mcal{K}},\qquad f,g\in\mcal{K}.
\end{equation*}
\end{Lemma}
We will give a proof of Lemma~\ref{lem:quasi-free_covariance} in Section~\ref{sect:existence_PfPP} for readers' convenience.

As we have announced, we work on the case when $\mcal{K}=\ell^{2}(\mfrak{X})\oplus\ell^{2}(\mfrak{X})$ equipped with $\Gamma$ given by (\ref{eq:Gamma}), and will adopt this pair $(\mcal{K},\Gamma)$ in the sequel without any specification.
The associated CAR algebra $\mcal{C}(\mcal{K},\Gamma)$ is generated by $a_{x}=B((0,e_{x}))$ and $a_{x}^{\ast}=B((e_{x},0))$, $x\in\mfrak{X}$.
Notice that the notation is compatible with the $\ast$-involution since $\Gamma (e_{x},0)=(0,e_{x})$.
Then, from the definition of a quasi-free state, we have
\begin{equation}
\label{eq:quasifree_correlation_function}
	\varphi_{S}\big(a_{x_{1}}^{\ast} \cdots a_{x_{n}}^{\ast}a_{x_{n}}\cdots a_{x_{1}}\big)=\pf [\mbb{K}_{S}(x_{i},x_{j}) ]_{1\le i,j\le n},
\end{equation}
where the matrix-valued function $\mbb{K}_{S}(\cdot,\cdot)\colon \mfrak{X}\times\mfrak{X}\to M(2;\mbb{C})$ was defined in~(\ref{eq:matrix_kernel_S}) associated with the operator~$S$. In fact, the values of the matrix-valued function $\mbb{K}_{S}(\cdot,\cdot)$ is written in terms of the quasi-free state as
\begin{equation*}
	\mbb{K}_{S}(x,y)=\left(
	\begin{matrix}
	\varphi_{S}(a_{x}^{\ast}a_{y}^{\ast}) & \varphi_{S}(a_{x}^{\ast}a_{y}) \\
	\varphi_{S}(a_{x}a_{y}^{\ast})-\delta_{x,y} & \varphi_{S}(a_{x}a_{y})
	\end{matrix}
	\right),\qquad x,y\in\mfrak{X}.
\end{equation*}
At this stage, we can find an analogy between this expectation value and a correlation function of a PfPP.

\begin{Remark}\label{rem:equivalence_CAR}
We may define functionals $a^{\ast}\colon \ell^{2}(\mfrak{X})\to \mcal{C}(\mcal{K},\Gamma)$ and $a\colon \ell^{2}(\mfrak{X})\to \mcal{C}(\mcal{K},\Gamma)$ by extending the assignments $e_{x}\mapsto a_{x}^{\ast}$ and $e_{x}\mapsto a_{x}$, $x\in\mfrak{X}$ linearly and anti-linearly, respectively. In other words, these functionals are defined by $a^{\ast}(u)=B((u,0))$, $a(u)=B((0,Ju))$, $u\in \ell^{2}(\mfrak{X})$.
Then, we find anti-commutation relations that are widely accepted~\cite{BratteliRobinson1997}:
\begin{gather}\label{eq:anti-commutation_BratteliRobinson}
	 \{a^{\ast}(u),a^{\ast}(v)\}=\{a(u),a(v)\}=0, \qquad
	 \{a(u),a^{\ast}(v)\}=(u,v)_{\ell^{2}(\mfrak{X})}, \qquad u,v\in \ell^{2}(\mfrak{X}).
\end{gather}
Conversely, when we have linear and anti-linear functional $a^{\ast}$, $a$ satisfying the anti-commutation relations (\ref{eq:anti-commutation_BratteliRobinson}), the generators of $\mcal{C}(\mcal{K},\Gamma)$ are recovered as $B((u,v))=a^{\ast}(u)+a(Jv)$, $(u,v)\in\mcal{K}$.
\end{Remark}

Let us consider a commutative algebra topologically generated by $a_{x}^{\ast}a_{x}$, $x\in\mfrak{X}$, which is identified with the algebra $C(\Omega)$ of continuous functions on $\Omega$ by the correspondence
\begin{equation*}
	\prod_{i=1}^{n}a_{x_{i}}^{\ast}a_{x_{i}}\mapsto \chi_{\Omega_{x_{1},\dots, x_{n}}},\qquad x_{1},\dots, x_{n}\in\mfrak{X}\colon \ \mbox{distinct.}
\end{equation*}
In the sequel, we regard $C(\Omega)$ as a subalgebra of $\mcal{C}(\mcal{K},\Gamma)$ under this correspondence.

Our strategy to prove Proposition~\ref{prop:existence_ppp} is to identify the above expectation value~(\ref{eq:quasifree_correlation_function}) with the correlation function of the desired PfPP, expecting that for a quasi-free state $\varphi_{S}$, there exists a probability measure $M^{S}$ on $(\Omega,\Sigma)$ so that the restriction of $\varphi_{S}$ on the subalgebra $C(\Omega)$ is identical to the integration with respect to $M^{S}$:
\begin{equation*}
	\varphi_{S}(F)=\int_{\Omega}F(\omega)M^{S}({\rm d}\omega),\qquad F\in C(\Omega).
\end{equation*}
Then, the probability measure $M^{S}$ is automatically a PfPP.
We will see in Section~\ref{sect:existence_PfPP} that this indeed happens to prove Proposition \ref{prop:existence_ppp}.

\begin{Remark}
It seems highly nontrivial if our construction can be extended to the case of continuous systems. If the space is continuous, the {\it field operators} $a(x)$, $a^{\ast}(x)$ labeled by points $x$ in the space are no longer operators. Instead, they are operator-valued distributions. Hence, for a suitable test function $u$, their integral $a(u)=\int u(x)a(x){\rm d}x$, $a^{\ast}(u)=\int u(x)a^{\ast}(x){\rm d}x$ make sense as operators. On the other hand, in the context of point processes, it is natural to consider an integral of the form $\rho (u)\overset{?}{=}\int u(x)a^{\ast}(x)a(x){\rm d}x$ for a test function~$u$. In fact, the existence of such operators allows us to have
\begin{equation*}
	\varphi_{S} (\rho (u_{1})\cdots \rho (u_{n}) )\overset{?}{=}\int \left(\prod_{i=1}^{n}{\rm d}x_{i}\right)u_{1}(x_{1})\cdots u_{n}(x_{n})\pf [\mbb{K}_{S}(x_{i},x_{j})]_{1\le i,j\le n}
\end{equation*}
under a quasi-free state $\varphi_{S}$ and to identify the integrated Pfaffian as a correlation function of a~PfPP. It is, however, not ensured that the integral $\rho (u)$ does make sense in general.
\end{Remark}

\subsection{Perfectness of a probability measure}
As we have seen in Lemma \ref{lem:quasi-free_covariance}, the set $\mcal{Q}(\mcal{K},\Gamma)$ labels quasi-free states over $\mcal{C}(\mcal{K},\Gamma)$. In this sense, Proposition~\ref{prop:existence_ppp} states that associated to a quas-free state, a PfPP exists.
Recently, Olshanski \cite{Olshanski2020} proposed a scheme to invert this correspondence, which is outlined here.

Let $\mfrak{S}$ be the group of finite permutations of $\mfrak{X}$.
The assumptions are
\begin{enumerate}\itemsep=0pt
\item 	A probability measure $M$ on $(\Omega,\Sigma)$ is $\mfrak{S}$-quasi-invariant.
\item 	The set $\mfrak{X}$ is equipped with a linear order $\le$ so that the ordered set $(\mfrak{X},\le)$ is isomorphic to $\mbb{Z}$ or $\mbb{N}$.
\end{enumerate}
We consider the gauge invariant subalgebra $\mcal{A}(\mcal{K},\Gamma)\subset \mcal{C}(\mcal{K},\Gamma)$ topologically generated by $a_{x}^{\ast}a_{y}$, $x,y\in\mfrak{X}$.
Under the above assumptions, we can associate to the probability measure~$M$ a~representation $\mcal{T}^{M}$ of the gauge invariant subalgebra $\mcal{A}(\mcal{K},\Gamma)$ on the Hilbert space $L^{2}(\Omega, M)$. Then, we immediately obtain a state~$\varphi^{M}$ on $\mcal{A}(\mcal{K},\Gamma)$ by
\begin{equation*}
	\varphi^{M}(A)=\big(\mbb{I},\mcal{T}^{M}(A)\mbb{I}\big)_{L^{2}(\Omega,M)},\qquad A\in \mcal{A}(\mcal{K},\Gamma),
\end{equation*}
where $\mbb{I}\in L^{2}(\Omega,M)$ is the unit constant function on~$\Omega$.
By construction of the representa\-tion~$\mcal{T}^{M}$, the action of $C(\Omega)$ on $L^{2}(\Omega,M)$ is just given by the multiplication. Therefore, we have
\begin{equation*}
	\varphi^{M}(F)= (\mbb{I},F\mbb{I} )_{L^{2}(\Omega,M)}=\int_{\Omega}F(\omega)M({\rm d}\omega),\qquad F \in C(\Omega),
\end{equation*}
which implies that the state $\varphi^{M}$ restricted on the commutative subalgebra $C(\Omega)$ is just the expectation value with respect to the probability measure~$M$ and, in particular, if~$M$ is a PfPP, $\varphi^{M}$ restricted on~$C(\Omega)$ admits a Pfaffian expression.

\begin{Definition}
Let $M$ be a $\mfrak{S}$-quasi-invariant probability measure on $(\Omega,\Sigma)$ and assume that~$\mfrak{X}$ can be equipped with a linear order $\le$ so that $(\mfrak{X},\le)$ is isomorphic to $\mbb{Z}$ or $\mbb{N}$. The probability measure~$M$ is said to be {\it perfect} if there exists a quasifree state $\varphi$ on $\mcal{C}(\mcal{K},\Gamma)$ such that the resulting state $\varphi^{M}$ on $\mcal{A}(\mcal{K},\Gamma)$ is realized as
\begin{equation*}
	\varphi^{M}=\varphi|_{\mcal{A}(\mcal{K},\Gamma)}.
\end{equation*}
\end{Definition}

\subsection{Schur measures} Schur measures form a family of DPPs on $\mfrak{X}=\mbb{Z}+\frac{1}{2}$ introduced in~\cite{Okounkov2001} that includes the Plancherel measure and the $z$-measure as special cases.
Let $\mbb{Y}$ be the collection of partitions, each element of which is a sequence of non-increasing integers $\lambda=(\lambda_{1}\ge \lambda_{2}\ge \cdots \ge 0)$ such that there exists $\ell\in\mbb{N}$ and $\lambda_{\ell+1}=0$. To each $\lambda\in\mbb{Y}$, we associate a subset
\begin{equation*}
	\mbb{M}(\lambda)=\left\{\lambda_{i}-i+\frac{1}{2}\,\bigg|\,i\in\mbb{N}\right\}\subset\mfrak{X},
\end{equation*}
which defines an embedding $\mbb{M}\colon \mbb{Y}\hookrightarrow \Omega$. Therefore, given a probability measure on $\mbb{Y}$, we obtain one on $(\Omega,\Sigma)$ by pushing it forward via $\mbb{M}$.

Let $\mbb{T}$ be the collection of data $\rho=(\alpha;\beta)$, where $\alpha=(\alpha_{1}\ge \alpha_{2}\ge\cdots \ge 0)$ and $\beta=(\beta_{1}\ge \beta_{2}\ge \cdots \ge 0)$ satisfying $\sum_{j\ge 1}\alpha_{j}+\sum_{j\ge 1}\beta_{j}\le 1$.
It is known that $\mbb{T}$ parametrizes Schur-positive specializations of the ring of symmetric functions.

We introduce a subset $\mbb{T}^{\circ}$ consisting of $\rho=(\alpha;\beta)$ such that $\alpha_{1}<1$ and $\beta_{1}<1$.
Note that the difference $\mbb{T}\backslash\mbb{T}^{\circ}$ consists of two elements $(1,0,\dots; 0,0,\dots)$ and $(0,0,\dots; 1,0,\dots)$. The Schur measure $M_{s(\rho)}$ associated with $\rho\in\mbb{T}^{\circ}$ is defined by
\begin{equation*}
	M_{s(\rho)}(\omega)\propto\begin{cases}
		s_{\lambda}(\rho)^{2}, & \omega=\mbb{M}(\lambda),\\
		0, & \mrm{otherwise},
	\end{cases}\qquad \omega\in \Omega,
\end{equation*}
where $s_{\lambda}(\rho)$ is the Schur function associated with $\lambda\in\mbb{Y}$ specialized at $\rho$. Note that, in general, we are allowed to have a weight at $\omega=\mbb{M}(\lambda)$ proportional to $s_{\lambda}(\rho)s_{\lambda}(\pr{\rho})$ with possibly different specializations $\rho$ and $\pr{\rho}$, but we concentrate on the spacial case when $\rho=\pr{\rho}$.

\begin{Theorem}\label{thm:perfectness_Schur}
Assume that a Schur measure $M_{s(\rho)}$ for $\rho\in\mbb{T}^{\circ}$ is $\mfrak{S}$-quasi invariant. Then it is perfect.
\end{Theorem}

\subsection{Conditional measures of PfPPs}
Let $X$ and $\pr{X}$ be disjoint finite subsets in $\mfrak{X}$ and take a cylinder set
\begin{equation*}
	C(X,\pr{X})=\{\omega\in \Omega\,|\, X\subset \omega,\, \pr{X}\cap\omega=\varnothing\},
\end{equation*}
which consists of configurations such that every points in $X$ are occupied and those in $\pr{X}$ are unoccupied.
We identify the cylinder set $C(X,\pr{X})$ with $\Omega (\mfrak{X}\backslash (X\sqcup\pr{X})):=\{0,1\}^{\mfrak{X}\backslash (X\sqcup\pr{X})}$ via the map
\begin{equation*}
	F_{X,\pr{X}}\colon \ C(X,\pr{X})\to \Omega (\mfrak{X}\backslash (X\sqcup\pr{X}));\qquad \omega\mapsto \omega\backslash X.
\end{equation*}
For a probability measure $M$ on $(\Omega,\Sigma)$, assume that the cylinder set has strictly positive weight: $M(C(X,\pr{X}))>0$.
We define the conditional measure of $M$ on $(X,\pr{X})$ by
\begin{equation*}
	M_{X,\pr{X}}:=(F_{X,\pr{X}})_{\ast}\frac{M\big|_{C(X,\pr{X})}}{M(C(X,\pr{X}))}.
\end{equation*}

We focus on PfPPs associated with projection operators.
For a subset $A\subset\mfrak{X}$, we set
\begin{equation*}
	\mcal{K}_{A}^{+}=\overline{\mrm{Span}\{(e_{x},0)\,|\,x\in A\}},\qquad \mcal{K}^{-}_{A}=\overline{\mrm{Span}\{(0,e_{x})\,|\,x\in A\}}
\end{equation*}
and $\mcal{K}_{A}=\mcal{K}_{A}^{+}\oplus\mcal{K}_{A}^{-}$, which is the direct sum of Hilbert spaces.
It is obvious from the definition that $\Gamma$ preserves the subspace $\mcal{K}_{A}$. Thus, we can write $\Gamma_{A}:=\Gamma|_{\mcal{K}_{A}}$.

Let $X$ and $\pr{X}$ be finite disjoint subsets of $\mfrak{X}$ and take a projection operator $P\in \mrm{Gr}(\mcal{K},\Gamma)$.
Then, $P\mcal{K}\subset\mcal{K}$ is a closed subspace. We define a new closed subspace
\begin{equation*}
	P\mcal{K}_{X,\pr{X}}:=\big(P\mcal{K}+\big(\mcal{K}_{X}^{+}\oplus \mcal{K}^{-}_{\pr{X}}\big)\big)\cap \mcal{K}_{\mfrak{X}\backslash (X\sqcup\pr{X})}
\end{equation*}
in $\mcal{K}_{\mfrak{X}\backslash (X\sqcup\pr{X})}$ and a projection operator $P_{X,\pr{X}}$ as the orthogonal projection onto $P\mcal{K}_{X,\pr{X}}$ in $\mcal{K}_{\mfrak{X}\backslash (X\sqcup\pr{X})}$.

\begin{Lemma}\label{lem:conditional_projection}
Let $X$ and $\pr{X}$ be finite disjoint subsets of $\mfrak{X}$. Then, we have
\begin{equation*}
	P_{X,\pr{X}}\in\mrm{Gr}\big(\mcal{K}_{\mfrak{X}\backslash(X\sqcup\pr{X})},\Gamma_{\mfrak{X}\backslash (X\sqcup\pr{X})}\big).
\end{equation*}
\end{Lemma}

 Let us also introduce the notion of regularity.
 \begin{Definition} Let $X$ and $\pr{X}$ be finite disjoint subsets of $\mfrak{X}$.
 A projection operator $P\in \mrm{Gr}(\mcal{K},\Gamma)$ is said to be $(X,\pr{X})$-regular if
 \begin{equation*}
 	P\mcal{K}\cap \big(\mcal{K}_{X}^{+}\oplus \mcal{K}_{\pr{X}}^{-}\big)=\{0\}.
 \end{equation*}
 \end{Definition}

 The following result is an analogue of \cite[Proposition~6.13]{Olshanski2020} for PfPPs.
 \begin{Theorem}\label{thm:conditional_projection}
 Let $X$ and $\pr{X}$ be finite disjoint subsets of $\mfrak{X}$.
 Assume that a projection operator $P\in\mrm{Gr}(\mcal{K},\Gamma)$ is $(X,\pr{X})$-regular.
 Then, we have
 \begin{equation*}
 	M^{P}_{X,\pr{X}}:=\big(M^{P}\big)_{X,\pr{X}}=M^{P_{X,\pr{X}}}.
 \end{equation*}
 In particular, it is a PfPP associated with a projection operator.
 \end{Theorem}

 \begin{Remark}
A similar problem has been studied in~\cite{BufetovCundenQiu2019}. Our result describes the reduction of a projection operator, which is new.
 \end{Remark}

 A natural application of Theorem \ref{thm:conditional_projection} is a proof of quasi-invariance of PfPPs with respect to the symmetric group along the line of~\cite{BufetovOlshanski2019,Olshanski2011,Olshanski2020}, which is set aside for a future work.

 \subsection*{Organization}
 In Section~\ref{sect:existence_PfPP}, we recall the Fock representations of a CAR algebra and prove Propositions~\ref{prop:existence_ppp},~\ref{prop:existence_ppp_fredholm}, and~\ref{prop:finitary_PfPP}.
 In Section~\ref{sect:measre_to_state}, after recalling the procedure of obtaining a state over the gauge-invariant subalgebra of a CAR algebra from a probability measure proposed in~\cite{Olshanski2020}, we give a proof of Theorem~\ref{thm:perfectness_Schur}.
 Section~\ref{sect:conditional_measures} is devoted to a proof of Theorem~\ref{thm:conditional_projection}, which includes a one of Lemma~\ref{lem:conditional_projection} as a part.
 In Appendix~\ref{sect:shifed_Schur_measures}, we illustrate that the shifted Schur measures are understood as examples in our perspective from a CAR algebra.

\section{Existence of Pfaffian point processes} \label{sect:existence_PfPP}
\subsection{Fock representations}
Here, we see Fock representations of $\mcal{C}(\mcal{K},\Gamma)$, which play prominent roles in the theory of a CAR algebra.
\subsubsection{General construction}
Let $\mcal{H}$ be a Hilbert space.
For each $n\in\mbb{N}$, we write $\bigwedge^{n}\mcal{H}$ for the $n$-th wedge product of $\mcal{H}$, which is generated by the vectors $u_{1}\wedge\cdots \wedge u_{n}$, $u_{1},\dots, u_{n}\in\mcal{H}$ subject to the anti-symmetry:
\begin{equation*}
	u_{1}\wedge\cdots \wedge u_{n}=\operatorname{sgn} (\sigma) u_{\sigma(1)}\wedge\cdots \wedge u_{\sigma (n)},\qquad \sigma\in \mfrak{S}_{n}.
\end{equation*}
When we take $\{e_{i}\}_{i=1,2,\dots}$ for a complete orthonormal system of $\mcal{H}$, then the vectors $e_{i_{1}}\wedge\cdots \wedge e_{i_{n}}$, $i_{1}>\cdots >i_{n}$ form a complete orthonormal system of $\bigwedge^{n}\mcal{H}$.
The Fermi Fock space over $\mcal{H}$ is defined by
\begin{equation*}
	\mcal{F}(\mcal{H})=\bigoplus_{n=0}^{\infty}\bigwedge^{n}\mcal{H},
\end{equation*}
where we set $\bigwedge^{0}\mcal{H}=\mbb{C}\bm{1}$.
The Fermi Fock space admits the natural inner product induced from each component of the direct sum and becomes a Hilbert space, i.e., the direct sum is understood in the topological sense.

For each $u\in \mcal{H}$, the creation operator $a_{\mcal{H}}^{\ast}(u)$ is an operator on $\mcal{F}(\mcal{H})$ defined by
\begin{equation*}
	a_{\mcal{H}}^{\ast}(u)\colon \ \eta\mapsto u\wedge \eta,\qquad \eta\in \mcal{F}(\mcal{H}).
\end{equation*}
The annihilation operator $a_{\mcal{H}}(u)$ is defined as the adjoint operator of $a_{\mcal{H}}^{\ast}(u)$.
By definition, it is obvious that creation and annihilation operators act as
\begin{equation*}
	a_{\mcal{H}}^{\ast}(u)\colon \ \bigwedge^{n}\mcal{H}\to \bigwedge^{n+1}\mcal{H},\qquad a_{\mcal{H}}(u)\colon \ \bigwedge^{n}\mcal{H}\to \bigwedge^{n-1}\mcal{H}.
\end{equation*}
We can also see that the assignment $u\mapsto a_{\mcal{H}}^{\ast}(u)$ is linear and $u\mapsto a_{\mcal{H}}(u)$ is anti-linear.

\subsubsection{Fock representation and a quasi-free state}
Let us take a projection operator $P\in\mrm{Gr}(\mcal{K},\Gamma)$, associated to which we can construct a representation of $\mcal{C}(\mcal{K},\Gamma)$ on a Fock space $\mcal{F}(P\mcal{K})$.
To describe the action, notice that, from the property $1-P=\overline{P}$, we can see that the projection to the complementary subspace of $P\mcal{K}$ is $\overline{P}$.
Let us set
\begin{equation*}
	\pi_{P}\colon \ \mcal{C}(\mcal{K},\Gamma)\to \mfrak{B}(\mcal{F}(P\mcal{K}));\qquad B(f)\mapsto a_{P\mcal{K}}^{\ast}(Pf)+a_{P\mcal{K}}(P\Gamma f),\qquad f\in\mcal{K}.
\end{equation*}
Then, it is known that $(\pi_{P},\mcal{F}(P\mcal{K}))$ is a faithful and irreducible representation of $\mcal{C}(\mcal{K},\Gamma)$.
When we set
\begin{equation*}
	\varphi_{P}(A)=(\bm{1}, \pi_{P}(A)\bm{1})_{\mcal{F}(P\mcal{K})},\qquad A\in \mcal{C}(\mcal{K},\Gamma),
\end{equation*}
we can verify that $\varphi_{P}$ is exactly the quasi-free state corresponding to $P$ in the sense that it possesses the property required in Lemma \ref{lem:quasi-free_covariance}

In particular, when we take $P_{0}\in\mrm{Gr}(\mcal{K},\Gamma)$, we have $P_{0}\mcal{K}\simeq \ell^{2}(\mfrak{X})$, and the map $\pi_{P_{0}}$ is described as
\begin{equation*}
	\pi_{P_{0}}(a^{\ast}(u))=a_{P_{0}\mcal{K}}^{\ast}(u),\qquad \pi_{P_{0}}(a(u))=a_{P_{0}\mcal{K}}(u),\qquad u\in P_{0}\mcal{K}\simeq \ell^{2}(\mfrak{X}).
\end{equation*}

Let us describe a standard complete orthonormal system of the Fock space $\mcal{F}\big(\ell^{2}(\mfrak{X})\big)$. Since~$\mfrak{X}$ is countable, it can be equipped with a linear order~$\le$. For each $\omega=\{x_{1}>\cdots>x_{n}\}\in\Omega^{\circ}$, we set
\begin{equation*}
	e_{\omega}:=e_{x_{1}}\wedge\cdots \wedge e_{x_{n}}.
\end{equation*}
Then the collection $\{e_{\omega}\}_{\omega\in\Omega^{\circ}}$ forms a complete orthonormal system.

\begin{proof}[Proof of Lemma \ref{lem:quasi-free_covariance}]
Let $\varphi$ be a quasi-free state over $\mcal{C}(\mcal{K},\Gamma)$.
Then the assignment
\begin{equation*}
	Q_{\varphi}\colon \ \mcal{K}\times \mcal{K}\to\mbb{C};\qquad (f,g)\mapsto \varphi (B(f)^{\ast}B(g))
\end{equation*}
defines a quadratic form on $\mcal{K}$.
It follows from the relation in the CAR algebra that
\begin{equation*}
	B(f)^{\ast}B(f)\le B(f)^{\ast}B(f)+B(f)B(f)^{\ast}=\|f\|^{2},\qquad f\in\mcal{K}.
\end{equation*}
Since the norm of a state is unity, we have $Q_{\varphi}(f,f)\le \|B(f)^{\ast}B(f)\|\le \|f\|^{2}$ for any $f\in\mcal{K}$, which implies that $Q_{\varphi}$ is a bounded quadratic form. Therefore, owing to the correspondence between bounded operators and bounded quadratic forms (see, e.g., \cite[Chapter~1]{Koshmanenko1999}), there exists a bounded operator $S\in \mfrak{B}(\mcal{K})$ such that $Q_{\varphi}(f,g)=(f,Sg)_{\mcal{K}}$, $f,g \in \mcal{K}$. It also follows from the positivity of the state that the quadratic form $Q_{\varphi}$ is positive, and therefore, $S=S^{\ast}\ge 0$. Again, from the relation in the CAR algebra, we have
\begin{equation*}
	(f,Sg)_{\mcal{K}}=(f,g)_{\mcal{K}}-(\Gamma g, S\Gamma f)_{\mcal{K}}=(f,g)_{\mcal{K}}-\big(f,\overline{S} g\big)_{\mcal{K}},\qquad f,g\in\mcal{K},
\end{equation*}
which implies that $\Gamma S \Gamma=1-S$. Since $\Gamma S \Gamma\ge 0$, we can see that $S\le 1$.

Conversely, given an operator $S\in\mcal{Q}(\mcal{K},\Gamma)$, we define the following operator
\begin{equation*}
	P_{S}=\left(
	\begin{matrix}
	S & S^{1/2}(1-S)^{1/2} \\
	S^{1/2}(1-S)^{1/2} & 1-S
	\end{matrix}
	\right)
\end{equation*}
acting on $\what{\mcal{K}}=\mcal{K}\oplus\mcal{K}$. Notice that the assumption $0\le S=S^{\ast}\le 1$ ensures that the square roots~$S^{1/2}$ and $(1-S)^{1/2}$ make sense. Let us equip this Hilbert space with the anti-unitary involution
\begin{equation*}
	\what{\Gamma}=\left(
	\begin{matrix}
	\Gamma & 0 \\
	0 & -\Gamma
	\end{matrix}
	\right).
\end{equation*}
Then, it can be checked that $P_{S}\in \mrm{Gr}\big(\what{\mcal{K}},\what{\Gamma}\big)$, which implies that the functional $\varphi_{P_{S}}$ defined by
\begin{equation*}
	\varphi_{P_{S}}(A)= (\bm{1},\pi_{P_{S}}(A)\bm{1} )_{\mcal{F}(P_{S}\what{\mcal{K}})},\qquad A\in \mcal{C}\big(\what{\mcal{K}},\what{\Gamma}\big)
\end{equation*}
is a quasi-free state over $\mcal{C}\big(\what{\mcal{K}},\what{\Gamma}\big)$. Now, since $\what{\Gamma}$ acts diagonally along the direct sum decomposition, $\mcal{C}(\mcal{K},\Gamma)$ is regarded as a subalgebra of $\mcal{C}\big(\what{\mcal{K}},\what{\Gamma}\big)$ and the restriction $\varphi_{S}=\varphi_{P_{S}}|_{\mcal{C}(\mcal{K},\Gamma)}$ is the quasi-free state corresponding to the given~$S$.
\end{proof}

\subsection{Proof of Proposition \ref{prop:existence_ppp}}
Now, we are at the position of proving Proposition~\ref{prop:existence_ppp}.
Our strategy is to check the criteria by Lenard \cite{Lenard1975a, Lenard1975b}.

\subsubsection{Lenard's criteria}
Suppose that a system of functions $\big\{\rho_{n}\colon\mfrak{X}^{n}\to\mbb{R}\big\}_{n=1}^{\infty}$ is given.
A question is if there exists a~probability measure~$M$ on $(\Omega,\Sigma)$ such that its $n$-point correlation function $\rho^{M}_{n}$ coincides with the given function $\rho_{n}$ for every $n\in\mbb{N}$.
Lenard \cite{Lenard1975a, Lenard1975b} clarified the necessary and sufficient conditions for this to happen.

\begin{Theorem}[\cite{Lenard1975a, Lenard1975b}]
A system of functions $\big\{\rho_{n}\colon \mfrak{X}^{n}\to\mbb{R}\big\}_{n=1}^{\infty}$ is a one of correlation functions for a probability measure on $(\Omega,\Sigma)$ if it possesses the following properties:
\begin{enumerate}\itemsep=0pt
\item[$1.$] {\it Symmetry:} each function $\rho_{n}$ is a symmetric function, i.e.,
	\begin{equation*}
		\rho_{n}(x_{\sigma (1)},\dots, x_{\sigma (n)})=\rho_{n}(x_{1},\dots, x_{n})
	\end{equation*}
	for arbitrary $\sigma\in\mfrak{S}_{n}$ and $x_{1},\dots, x_{n}\in\mfrak{X}$.
\item[$2.$] {\it Positivity:} for any system of functions $\bm{\Phi}=\big\{\Phi_{n}\colon \mfrak{X}^{n}\to\mbb{R}\big\}_{n=0}^{N}$ $(\Phi_{0}$ is understood as a~constant$)$ with finite support such that
	\begin{equation}\label{eq:positivity_test_function}
		\Phi_{0}+\sum_{n=1}^{N}\sum_{\substack{i_{1},\dots,i_{n}\in I \\ \mrm{distinct}}}\Phi_{n}(x_{i_{1}},\dots, x_{i_{n}})\ge 0
	\end{equation}
	holds for any $\omega=\{x_{i}\}_{i\in I}\in\Omega$, we have
	\begin{equation*}
		\Phi_{0}+\sum_{n=1}^{N}\sum_{x_{1},\dots, x_{n}\in\mfrak{X}}\Phi_{n}(x_{1},\dots, x_{n})\rho_{n}(x_{1},\dots, x_{n})\ge 0.
	\end{equation*}
\end{enumerate}
\end{Theorem}

Therefore, Proposition \ref{prop:existence_ppp} reduces to the following assertion.
\begin{Lemma}
\label{lem:existence_ppp}
Let $S\in\mcal{Q}(\mcal{K},\Gamma)$. Then the system of functions
\begin{equation*}
	\rho_{n}(x_{1},\dots, x_{n})=\pf \left[\mbb{K}_{S}(x_{i},x_{j})\right]_{1\le i,j\le n}, \qquad x_{1},\dots, x_{n}\in\mfrak{X}, \qquad n\in\mbb{N}
\end{equation*}
satisfies the symmetry and positivity conditions.
\end{Lemma}

\subsubsection{Proof of Lemma \ref{lem:existence_ppp}}
As we have seen, each correlation function is realized as
\begin{equation*}
	\rho_{n}(x_{1},\dots, x_{n})=\varphi_{S}\big(a_{x_{1}}^{\ast}\cdots a_{x_{n}}^{\ast}a_{x_{n}}\cdots a_{x_{1}}\big).
\end{equation*}
Then, the symmetry condition is obviously satisfied since
\begin{equation*}
	a_{x_{1}}^{\ast}\cdots a_{x_{n}}^{\ast}a_{x_{n}}\cdots a_{x_{1}}
	=a_{x_{\sigma (1)}}^{\ast}\cdots a_{x_{\sigma (n)}}^{\ast}a_{x_{\sigma (n)}}\cdots a_{x_{\sigma (1)}}
\end{equation*}
holds for all $\sigma\in\mfrak{S}_{n}$.

To show the positivity condition, let $\bm{\Phi}=\big\{\Phi_{n}\colon\mfrak{X}^{n}\to\mbb{R}\big\}_{n=0}^{N}$ be a system of functions satisfying~(\ref{eq:positivity_test_function}).
We can see that
\begin{equation*}
	\Phi_{0}+\sum_{n=1}^{N}\sum_{x_{1},\dots, x_{n}\in\mfrak{X}}\Phi_{n}(x_{1},\dots, x_{n})\rho_{n}(x_{1},\dots, x_{n})
	=\varphi_{S} (A_{\bm{\Phi}} ),
\end{equation*}
where we set
\begin{equation*}
	A_{\bm{\Phi}}=\Phi_{0}+\sum_{n=1}^{N}\sum_{x_{1},\dots,x_{n}\in\mfrak{X}}\Phi_{n}(x_{1},\dots,x_{n})a_{x_{1}}^{\ast}\cdots a_{x_{n}}^{\ast}a_{x_{n}}\cdots a_{x_{1}}.
\end{equation*}
Therefore, owing to the positivity of a state over a C$^{\ast}$-algebra, it suffices to show that $A_{\bm{\Phi}}\in\mcal{C}(\mcal{K},\Gamma)$ is a positive element, which can be checked in any faithful representation. In fact, when we take a faithful representation $(\pi,\mcal{H})$, we may regard $\mcal{C}(\mcal{K},\Gamma)$ as a C$^{\ast}$-subalgebra of $\mfrak{B}(\mcal{H})$ via the embedding~$\pi$. Since the spectrum of an element in a C$^{\ast}$-subalgebra coincides with that in a~whole algebra (see, e.g., \cite[Proposition~2.2.7]{BratteliRobinson1987}), the relevant element $A_{\bm{\Phi}}\in\mcal{C}(\mcal{K},\Gamma)$ is positive if and only if $\pi \left(A_{\bm{\Phi}}\right)$ is a positive operator on $\mcal{H}$.

We can, in particular, take a Fock representation $\big(\pi_{P_{0}},\mcal{F}\big(\ell^{2}(\mfrak{X})\big)\big)$.
\begin{Lemma}
In the Fock space $\mcal{F}\big(\ell^{2}(\mfrak{X})\big)$, the complete orthonormal system $\{e_{\omega}\}_{\omega\in\Omega^{\circ}}$ diagonalizes $\pi_{P_{0}} (A_{\bm{\Phi}} )$ so that
\begin{equation*}
	\pi_{P_{0}}\left(A_{\bm{\Phi}}\right) e_{\omega}=\left(\Phi_{0}+\sum_{k=1}^{N}\sum_{\substack{i_{1},\dots,i_{k}=1\\ \mrm{distinct}}}^{n}\Phi_{k}(x_{i_{1}},\dots, x_{i_{k}})\right) e_{\omega},\qquad \omega=\{x_{1}>\cdots >x_{n}\},
\end{equation*}
\end{Lemma}
\begin{proof}
This follows from a direct computation.
Let us notice that $a_{x_{1}}^{\ast}\cdots a_{x_{n}}^{\ast}a_{x_{n}}\cdots a_{x_{1}}=0$ if any two points from $x_{1},\dots, x_{n}$ coincide. Hence, we have
\begin{equation*}
	A_{\bm{\Phi}}=\Phi_{0}+\sum_{n=1}^{N}\sum_{\substack{x_{1},\dots,x_{n}\in\mfrak{X}\\ \mrm{distinct}}}\Phi_{n}(x_{1},\dots,x_{n})\prod_{i=1}^{n}a_{x_{i}}^{\ast}a_{x_{i}}.
\end{equation*}
We can also see that $\pi_{P_{0}}(a_{x}^{\ast}a_{x})e_{\omega}=\chi_{[x\in\omega]}e_{\omega}$, $x\in\mfrak{X}$, $\omega\in\Omega^{\circ}$. Therefore, the desired result is obtained.
\end{proof}

The eigenvalues of $\pi_{P_{0}} (A_{\bm{\Phi}} )$ are non-negative from the assumption implying that $A_{\bm{\Phi}}$ is a positive element, and therefore, the system of functions $\{\rho_{n}\}_{n=1}^{\infty}$ fulfills the positivity conditions.
Now, the proof is complete.

\subsection{Restatement in terms of the Fredholm Pfaffian}
Here, we see the equivalence between Propositions \ref{prop:existence_ppp} and \ref{prop:existence_ppp_fredholm}.
First, notice that a multiplicative functional $\Psi_{\alpha}$ associated with a function $\alpha$ is identified with
\begin{equation*}
	\Psi_{\alpha}=\prod_{x\in\mfrak{X}}\left(\alpha (x)a_{x}^{\ast}a_{x}+a_{x}a_{x}^{\ast}\right)
\end{equation*}
in $C(\Omega)\subset \mcal{C}(\mcal{K},\Gamma)$. Since $\alpha-1$ is finitely supported and $\{a_{x},a_{x}^{\ast}\}=1$, factors except for finitely many ones in the product are unity. We can see that $\Psi_{\alpha}$ can also be expressed as
\begin{equation}
\label{eq:multiplicative_functional_expansion}
	\Psi_{\alpha}=\prod_{x\in\mfrak{X}}\big((\alpha(x)-1)a_{x}^{\ast}a_{x}+1\big)=1+\sum_{X\subset\mfrak{X}}\prod_{x\in X}(\alpha (x)-1)a_{x}^{\ast}a_{x}.
\end{equation}
Therefore, it is immediate that, for the the probability measure $M^{S}$ associated with $S\in\mcal{Q}(\mcal{K},\Gamma)$ in the sense of Proposition \ref{prop:existence_ppp}, the expectation value of $\Psi_{\alpha}$ is
\begin{equation*}
	\int_{\Omega}\Psi_{\alpha}(\omega)M^{S}({\rm d}\omega)=1+\sum_{X\subset\mfrak{X}}\prod_{x\in X}(\alpha (x)-1) \pf [\mbb{K}_{S}(x,y) ]_{x,y\in X}.
\end{equation*}
When we write ${\bf D}_{\sqrt{\alpha-1}}\colon \mfrak{X}\times\mfrak{X}\to M(2;\mbb{C})$ for the matrix-value function
\begin{equation*}
	{\bf D}_{\sqrt{\alpha-1}}(x,y):=\delta_{x,y}\left(
		\begin{matrix}
			\sqrt{\alpha (x)-1} & 0 \\
			0 & \sqrt{\alpha (x)-1}
		\end{matrix}
	\right),
\end{equation*}
we have
\[ \prod_{x\in X}(\alpha(x)-1)=\det [{\bf D}_{\sqrt{\alpha-1}}(x,y)]_{x,y\in X}.\] Due to the formula $\pf \big(B^{\mrm{T}}AB\big)=(\det B)(\pf A)$ for an anti-symmetric matrix $A$ and an arbitrary matrix $B$ of the same size, we have
\begin{equation*}
	\int_{\Omega}\Psi_{\alpha}(\omega)M^{S}({\rm d}\omega)=1+\sum_{X\subset\mfrak{X}}\pf\big[\big(\sqrt{\alpha -1}\mbb{K}_{S}\sqrt{\alpha -1}\big)(x,y)\big]_{x,y\in X}=\pf [\mbb{J}+\mbb{K}_{S} ]_{\mfrak{X}}
\end{equation*}
as has also been shown in~\cite{Rains2000}.

Conversely, let $M^{S}$ be the probability measure associated with $S\in \mcal{Q}(\mcal{K},\Gamma)$ in the sense of Proposition~\ref{prop:existence_ppp_fredholm}.
For a finite subset $X\subset\mfrak{X}$, we set $\alpha_{X}=\chi_{X}+1$, where $\chi_{X}$ is the characteristic function of~$X$. Then $\alpha_{X}-1$ is finitely supported and $\alpha(x)\ge 1$, $x\in\mfrak{X}$. Owing to the expression~(\ref{eq:multiplicative_functional_expansion}), we have
\begin{equation*}
	\Psi_{\alpha_{X}}=1+\sum_{Y=\{y_{1},\dots, y_{n}\}\subset X}\chi_{\Omega_{y_{1},\dots, y_{n}}},
\end{equation*}
and therefore,
\begin{equation*}
	\int_{\Omega}\Psi_{\alpha_{X}}(\omega)M^{S}({\rm d}\omega)=1+\sum_{Y=\{y_{1},\dots, y_{n}\}\subset X}\rho^{M^{S}}(y_{1},\dots, y_{n}).
\end{equation*}
On the other hand, from the characterization of $M^{S}$, it follows that
\begin{equation*}
	\int_{\Omega}\Psi_{\alpha_{X}}(\omega) M^{S}({\rm d}\omega)=1+\sum_{Y\subset X}\pf [\mbb{K}_{S}(x,y) ]_{x,y\in Y}.
\end{equation*}
If we define vectors
\begin{gather*}
	{\bf v}_{S} =\big(\rho^{M^{S}}(x_{1},\dots, x_{n})\big)_{X=\{x_{1},\dots, x_{n}\}\subset\mfrak{X}},\qquad {\bf w}_{S} =\big(\pf [\mbb{K}_{S}(x,y) ]_{x,y\in X}\big)_{X\subset\mfrak{X}}
\end{gather*}
with understanding $ ({\bf v}_{S} )_{\varnothing}= ({\bf w}_{S} )_{\varnothing}=1$ and a matrix $A=(A_{X,Y})_{X, Y\subset \mfrak{X}}$ by
\begin{equation*}
	A_{X,Y}=\begin{cases}
	1, & X\supset Y, \\
	0, & \mbox{otherwise},
	\end{cases}
\end{equation*}
we have $A{\bf v}_{S}=A{\bf w}_{S}$. Since the matrix $A$ is triangular with respect to the partial order induced from the inclusion relation with unit diagonal, it is invertible. Therefore, ${\bf v}_{S}={\bf w}_{S}$ implying that
\begin{equation*}
	\rho^{M^{S}}(x_{1},\dots, x_{n})=\pf [\mbb{K}_{S}(x_{i},x_{j}) ]_{1\le i,j\le n},\qquad x_{1},\dots, x_{n}\in\mfrak{X}\colon \ \mbox{distinct},
\end{equation*}
which is the desired property.

\subsection{Unitarily implementable Bogoliubov automorphisms}
\subsubsection{Bogoliubov automorphisms}
Let us consider the following collection of operators:
\begin{equation*}
	\mcal{I}(\mcal{K},\Gamma)=\big\{V\in \mfrak{U}(\mcal{K})\,|\, \overline{V}=V\big\},
\end{equation*}
where $\mfrak{U}(\mcal{K})$ is the set of unitary operators on $\mcal{K}$. It obviously forms a group.
Given an operator $V\in \mcal{I}(\mcal{K},\Gamma)$, we can define an automorphism $\alpha_{V}$ of $\mcal{C}(\mcal{K},\Gamma)$ by $\alpha_{V}(B(f)):=B(Vf)$, $f\in \mcal{K}$, which is called the Bogoliubov automorphism associated with $V$.
When we have a state~$\varphi$ over~$\mcal{C}(\mcal{K},\Gamma)$, we can twist it by a Bogoliubov automorphism to obtain a new state $\varphi\circ \alpha_{V}$, which defines a right action of the group $\mcal{I}(\mcal{K},\Gamma)$ on the collection of states. On the other hand, the group $\mcal{I}(\mcal{K},\Gamma)$ also acts on $\mcal{Q}(\mcal{K},\Gamma)$ from the right via
\begin{equation*}
	\mcal{Q}(\mcal{K},\Gamma)\times \mcal{I}(\mcal{K},\Gamma)\to \mcal{Q}(\mcal{K},\Gamma);\qquad (S,V)\mapsto V^{\ast}SV.
\end{equation*}
When the state is quasi-free, then these two actions are compatible:
\begin{Lemma}
For $S\in \mcal{Q}(\mcal{K},\Gamma)$ and $V\in \mcal{I}(\mcal{K},\Gamma)$, we have $\varphi_{S}\circ \alpha_{V}=\varphi_{V^{\ast}SV}$.
\end{Lemma}
\begin{proof}
It is easily checked that
\begin{equation*}
	\varphi_{S}\circ \alpha_{V}(B(f)^{\ast}B(g))=\varphi_{S}(B(Vf)^{\ast}B(Vg))=(Vf,SVg)_{\mcal{K}}=\varphi_{V^{\ast}SV}(B(f)^{\ast}B(g))
\end{equation*}
for any $f,g\in\mcal{K}$, which implies the desired equality.
\end{proof}

It is obvious that the group action by $\mcal{I}(\mcal{K},\Gamma)$ preserves the collection $\mrm{Gr}(\mcal{K},\Gamma)$ of projection operators. Moreover, we have the following:
\begin{Lemma}
The group $\mcal{I}(\mcal{K},\Gamma)$ acts on $\mrm{Gr}(\mcal{K},\Gamma)$ transitively.
\end{Lemma}
\begin{proof}
For projection operators $P, \pr{P}\in \mrm{Gr}(\mcal{K},\Gamma)$, let us take complete orthonormal systems $\{f_{i}\}_{i\in I}$ and $\{g_{i}\}_{i\in I}$ of $P\mcal{K}$ and $\pr{P}\mcal{K}$, respectively. Then $\{f_{i},\Gamma f_{i}\}_{i\in I}$ and $\{g_{i},\Gamma g_{i}\}_{i\in I}$ are both complete orthonormal systems of $\mcal{K}$. If we define an operator $V$ by
\begin{equation*}
	V\sum_{i\in I} (a_{i}f_{i}+b_{i}\Gamma f_{i} ):=\sum_{i\in I}(a_{i}g_{i}+b_{i}\Gamma g_{i}),\qquad a_{i}, b_{i}\in\mbb{C},\qquad i\in I,
\end{equation*}
it is a unitary operator such that $V^{\ast}\pr{P}V=P$ and commutes with $\Gamma$.
\end{proof}

\subsubsection{Unitary implementability}
Let $P\in \mrm{Gr}(\mcal{K},\Gamma)$ be a projection and take an operator $V\in \mcal{I}(\mcal{K},\Gamma)$.
We say that the Bogoliubov automorphism $\alpha_{V}$ is unitarily implementable on the Fock representation $(\pi_{P},\mcal{F}(P\mcal{K}))$ if there exists a unitary operator $\mcal{U}$ on $\mcal{F}(P\mcal{K})$ such that
\begin{equation*}
	\pi_{P}\circ \alpha_{V}=\mrm{Ad}(\mcal{U}^{\ast})\circ \pi_{P}.
\end{equation*}
Since $\alpha_{V}$ is an automorphism, $\pi_{P}\circ\alpha_{V}$ is an irreducible representation and is identified with the Fock representation $\pi_{V^{\ast}PV}$.
It is known that the unitary implementability is equivalent to the quasi-equivalence of two representations $\pi_{P}\circ \alpha_{V}$ and $\pi_{P}$, which is, therefore, equivalent to the quasi-equivalence of quasi-free states $\varphi_{V^{\ast}PV}$ and $\varphi_{P}$.

The following criterion is well-known:
\begin{Theorem}[\cite{Araki1970, PowersStormer1970, ShaleStinespring1965}]
\label{thm:implementability}
Let $P, \pr{P}\in \mrm{Gr}(\mcal{K},\Gamma)$ be projection operators and take an operator $V\in \mcal{I}(\mcal{K},\Gamma)$ such that $\pr{P}=V^{\ast}PV$.
Then $V$ is unitarily implementable on the Fock representation $(\pi_{P},\mcal{F}(P\mcal{K}))$ if and only if $P-\pr{P}$ is of Hilbert--Schmidt class.
\end{Theorem}

\subsubsection{Proof of Proposition~\ref{prop:finitary_PfPP}}
Let $P\in \mrm{Gr}(\mcal{K},\Gamma)$ be such that $P-P_{0}$ is of Hilbert--Schmidt class as assumed in Proposition~\ref{prop:finitary_PfPP}
and take $V\in \mcal{I}(\mcal{K},\Gamma)$ so that $P=V^{\ast}P_{0}V$.
Then, it follows from Theorem \ref{thm:implementability} that there exists a unitary operator $\mcal{U}$ on $\mcal{F}\big(\ell^{2}(\mfrak{X})\big)$ such that
\begin{equation*}
	\varphi_{P}(A)=(\bm{1},\pi_{P_{0}}(\alpha_{V}(A))\bm{1})_{\mcal{F}(\ell^{2}(\mfrak{X}))} =(\mcal{U}\bm{1},\pi_{P_{0}}(A)\mcal{U}\bm{1})_{\mcal{F}(\ell^{2}(\mfrak{X}))},\qquad A\in \mcal{C}(\mcal{K},\Gamma).
\end{equation*}
We may expand $\mcal{U}\bm{1}\in\mcal{F}\big(\ell^{2}(\mfrak{X})\big)$ in the complete orthonormal system $\{e_{\omega}\}_{\omega\in\Omega^{\circ}}$ as
\begin{equation*}
	\mcal{U}\bm{1}=\sum_{\omega\in\Omega^{\circ}} c^{P}(\omega)e_{\omega},\qquad c^{P}(\omega)\in\mbb{C},\qquad \omega\in\Omega^{\circ}.
\end{equation*}
It is obvious that $\widetilde{M}^{P}:=\big|c^{P}(\cdot)\big|^{2}$ defines a probability measure supported on $\Omega^{\circ}$ such that
\begin{equation*}
	\int_{\Omega}F(\omega)\widetilde{M}^{P}({\rm d}\omega)=\varphi_{P}(F)=\int_{\Omega}F(\omega)M^{P}({\rm d}\omega),\qquad F\in C(\Omega).
\end{equation*}
Therefore, we can conclude that $\widetilde{M}^{P}=M^{P}$ and, in particular, $M^{P}$ is supported on $\Omega^{\circ}$.

\subsubsection{Straightforward generalization of Proposition \ref{prop:finitary_PfPP}}
The above proof suggests a straightforward generalization of Proposition \ref{prop:finitary_PfPP}.
Let us take a~subset $X\subset\mfrak{X}$ and write $\Omega_{X}^{\mrm{fin}}$ for the subset of $\Omega$ consisting of $\omega$ such that $(\mfrak{X}\backslash X)\cap \omega$ and $X\backslash\omega$ are both finite. Note that, if $X$ is a finite set, then $\Omega_{X}^{\mrm{fin}}=\Omega^{\circ}$. Let $P_{X}^{\mrm{fin}}$ be the orthogonal projection onto $\mcal{K}^{+}_{\mfrak{X}\backslash X}\oplus \mcal{K}^{-}_{X}$.
Then we have the following:
\begin{Proposition}
\label{prop:finitary_PfPP_generalization}
If $P\in \mrm{Gr}(\mcal{K},\Gamma)$ is such that $P-P_{X}^{\mrm{fin}}$ is of Hilbert--Schmidt class, then the associated PfPP $M^{P}$ is supported on $\Omega_{X}^{\mrm{fin}}$.
\end{Proposition}

\subsection{Example}
In view of Proposition \ref{prop:existence_ppp}, one may associate PfPPs to a {\it bilinear Hamiltonian system} of quantum physics.
Let us illustrate this correspondence with a simple example.
We assume that $\mfrak{X}=\mbb{Z}+\frac{1}{2}$, and take an even function $\Upsilon\colon \mfrak{X}\to\mbb{R}$ and an odd function $\Delta\colon \mfrak{X}\to \mbb{C}$.
We decompose the Hilbert space $\mcal{K}=\ell^{2}(\mfrak{X})\oplus\ell^{2}(\mfrak{X})$ into $\mcal{K}=\bigoplus_{x\in\mfrak{X}}\mcal{K}_{x}$ with
$\mcal{K}_{x}=\mbb{C}(e_{x},0)\oplus\mbb{C}(0,e_{-x})$, $x\in\mfrak{X}$, where, as usual, $\bigoplus$ is understood as the completed direct sum of Hilbert spaces.
Accordingly, we define an operator $H=\sum_{x\in\mfrak{X}}H_{x}$ by
\begin{equation*}
	H_{x}=\left(
	\begin{matrix}
	\Upsilon (x) & \Delta (x) \\
	\overline{\Delta}(x) & -\Upsilon (x)
	\end{matrix}
	\right)
\end{equation*}
with respect to the basis $ ((e_{x},0), (0,e_{-x}) )$, $x\in\mfrak{X}$.
It is readily to see that $H$ is self-adjoint and satisfies $\overline{H}=-H$.
Each $H_{x}$, $x\in\mfrak{X}$ has eigenvalues $\pm\lambda_{x}$ with $\lambda_{x}=\sqrt{\Upsilon (x)^{2}+|\Delta (x)|^{2}}$ and, as normalized eigenvectors corresponding to the eigenvalues $\pm\lambda_{x}$, we can take
\begin{equation*}
	\gamma^{\pm}_{x}=u^{\pm}_{x}(e_{x},0)+\overline{u^{\mp}_{-x}}(0,e_{-x}),
\end{equation*}
where
\begin{gather}\label{eq:example_wavefunction}
	u^{\pm}_{x} =\sqrt{\frac{1}{2}\left(1\pm \frac{\Upsilon(x)}{\sqrt{\Upsilon (x)^{2}+|\Delta (x)|^{2}}}\right)}{\rm e}^{{\rm i}\phi_{x}}, \qquad
	\phi_{x} =
	\begin{cases}
	0, &x>0, \\
	\arg \Delta (x), &x<0.
	\end{cases}
\end{gather}
Note that, though $\arg \Delta (x)$ is only defined modulo $2\pi$, ${\rm e}^{{\rm i}\phi_{x}}$ is still well-defined.
Under this setting, we consider the following series of examples.

\begin{Proposition}
For $\beta\in \mbb{R}$, we define $S_{\beta}:=\big(1+{\rm e}^{-\beta H}\big)^{-1}$ by means of the spectral decomposition of $H$. Then, we have $S_{\beta}\in \mcal{Q}(\mcal{K},\Gamma)$. The corresponding correlation kernel $\mbb{K}_{S_{\beta}}$ is given by
\begin{gather*}
	\mbb{K}_{S_{\beta}}(x,y)=\left(
	\begin{matrix}
		\frac{1-{\rm e}^{-\beta\lambda_{x}}}{1+{\rm e}^{-\beta\lambda_{x}}}\overline{u^{+}_{-x}u^{-}_{x}}\delta_{x+y,0} & \frac{|u^{-}_{x}|^{2}+{\rm e}^{-\beta\lambda_{x}}(1-|u^{-}_{x}|^{2})}{1+{\rm e}^{-\beta\lambda_{x}}}\delta_{x,y} \vspace{1mm}\\
		\frac{|u^{+}_{x}|^{2}-1+{\rm e}^{-\beta\lambda_{x}}(1-|u^{+}_{x}|^{2})}{1+{\rm e}^{-\beta\lambda_{x}}}\delta_{x,y} & \frac{1-{\rm e}^{-\beta\lambda_{x}}}{1+{\rm e}^{-\beta\lambda_{x}}}u^{+}_{x}u^{-}_{-x}\delta_{x+y,0}
	\end{matrix}
	\right), \qquad x,y\in \mfrak{X}.
\end{gather*}
Furthermore, the weak limits $S_{\infty}:=\lim\limits_{\beta\to\infty}S_{\beta}$ and $S_{-\infty}:=\lim\limits_{\beta\to-\infty}S_{\beta}$ exist and $S_{\pm\infty}\in\mrm{Gr}(\mcal{K},\Gamma)$.
The projection operators $S_{\infty}$ and $S_{-\infty}$ are the orthogonal projection onto the closed subspaces generated by $\big\{\gamma^{+}_{x}\colon x\in\mfrak{X}\}$ and $\big\{\gamma^{-}_{x}\colon x\in\mfrak{X}\}$, respectively.
\end{Proposition}

\begin{proof}For each $x\in \mfrak{X}$, we write $P_{x}$ for the orthogonal projection in $\mcal{K}_{x}$ onto the subspace spanned by $\gamma^{+}_{x}$.
Explicitly, it is the matrix
\begin{equation}\label{eq:projection_example}
	P_{x}=\left(
	\begin{matrix}
		|u^{+}_{x}|^{2} & u^{+}_{-x}u^{-}_{x} \\
		\overline{u^{+}_{-x}u^{-}_{x}} & |u^{-}_{x}|^{2}
	\end{matrix}
	\right)
\end{equation}
in the basis $((e_{x},0), (0,e_{-x}))$.
Then, the operator admits the spectral decomposition
\begin{equation}\label{eq:example_KMS_spectral_decomposition}
	S_{\beta}=\sum_{x\in\mfrak{X}}\big(\big(1+{\rm e}^{-\beta\lambda_{x}}\big)^{-1}P_{x}+\big(1+{\rm e}^{\beta\lambda_{x}}\big)^{-1}(I_{x}-P_{x})\big),
\end{equation}
where $I_{x}$ is the identity operator on $\mcal{K}_{x}$, $x\in\mfrak{X}$.
Noting that $\Gamma\colon \mcal{K}_{x}\to\mcal{K}_{-x}$, we observe from~(\ref{eq:example_wavefunction}) and~(\ref{eq:projection_example}) that $\overline{P_{-x}}=I_{x}-P_{x}$.
Since the assignment $x\mapsto \lambda_{x}$ is an even function on $\mfrak{X}$, we can see that
\begin{align*}
	\overline{S_{\beta}}
	&=\sum_{x\in\mfrak{X}}\big(\big(1+{\rm e}^{-\beta\lambda_{x}}\big)^{-1}(I_{x}-P_{x})+\big(1+{\rm e}^{\beta\lambda_{x}}\big)^{-1}P_{x}\big) \\
&=\sum_{x\in\mfrak{X}}\big(I_{x}-\big(1+{\rm e}^{-\beta\lambda_{x}}\big)^{-1}P_{x}-\big(1+{\rm e}^{\beta\lambda_{x}}\big)^{-1}(I_{x}-P_{x})\big)
	 =I-S_{\beta}.
\end{align*}
Notice that $S_{\beta}$ is clearly self-adjoint and satisfies $0\leq S\leq I$. Hence, we conclude that $S_{\beta}\in \mcal{Q}(\mcal{K},\Gamma)$.
It is a straightforward manipulation to derive the expression of $\mbb{K}_{S_{\beta}}$ from~(\ref{eq:matrix_kernel_S}), (\ref{eq:projection_example}) and~(\ref{eq:example_KMS_spectral_decomposition}).

Existence of the weak limits $S_{\pm\infty}$ is ensured by the spectral decomposition~(\ref{eq:example_KMS_spectral_decomposition}) and the dominated convergence theorem.
It is also obvious from construction that $S_{\infty}=\sum_{x\in\mfrak{X}}P_{x}$ is the projection onto the closed subspace generated by $\big\{\gamma_{x}^{+}\colon x\in\mfrak{X}\}$, and that $S_{-\infty}=\sum_{x\in\mfrak{X}}(I_{x}-P_{x})=I-S_{\infty}$ is that onto the closed subspace generated by $\big\{\gamma_{x}^{-}\colon x\in\mfrak{X}\big\}$.
\end{proof}

\begin{Remark}
The property $\overline{H}=-H$ of the operator $H$ ensures that the family $\big\{{\rm e}^{{\rm i}tH}\colon t\in\mbb{R}\big\}$ is a one-parameter subgroup of $\mcal{I}(\mcal{K},\Gamma)$. Hence, it generates a one-parameter group of Bogoliubov automorphisms $\alpha_{{\rm e}^{{\rm i}tH}}$, $t\in\mbb{R}$ to give the CAR algebra $\mcal{C}(\mcal{K},\Gamma)$ a structure of a C$^{\ast}$-dynamical system.
Given a C$^{\ast}$-dynamical system, it is natural to study its ground states and Kubo--Martin--Schwinger (KMS) states.
For these notions of a C$^{\ast}$-dynamical system, ground states and KMS states, we refer readers to~\cite[Chapter~5.2]{BratteliRobinson1997}.
In fact, under the correspondence in Lemma~\ref{lem:quasi-free_covariance} between quasi-free states and operators in $\mcal{Q}(\mcal{K},\Gamma)$, each $S_{\beta}$, $\beta\in \mbb{R}$ corresponds to the $\beta$-KMS state and~$S_{\infty}$ is the ground state.
\end{Remark}

Let us apply the criteria in Propositions \ref{prop:finitary_PfPP} and \ref{prop:finitary_PfPP_generalization} to the projection operators $S_{\infty}$ and $S_{-\infty}$.
\begin{Proposition}\label{eq:example_PfPP_ground_state}
Assume that $|\Delta (x)|/\Upsilon (x)=o\big(|x|^{-1/2}\big)$ as $|x|\to\infty$
Then the PfPPs~$M^{S_{\infty}}$ and $M^{S_{-\infty}}$ are supported on $\Omega^{\circ}$ and $\Omega_{\mfrak{X}}^{\mrm{fin}}$, respectively.
In other words, if a point process~$X$ obeys~$M^{S_{\infty}}$, then we have $\# X<\infty$ almost surely, and if it obeys $M^{S_{-\infty}}$, we have $\# \mfrak{X}\backslash X<\infty$.
\end{Proposition}
\begin{proof}As before, we write $P_{0}$ for the orthogonal projection to the first component of $\mcal{K}=\ell^{2}(\mfrak{X})\oplus\ell^{2}(\mfrak{X})$.
Then, $P_{\mfrak{X}}^{\mrm{fin}}=I-P_{0}$ is the projection to the second component.
Due to Propositions~\ref{prop:finitary_PfPP} and~\ref{prop:finitary_PfPP_generalization}, our task is to estimate the difference
\begin{equation*}
	S_{\infty}-P_{0}=(I-S_{-\infty})-\big(I-P_{\mfrak{X}}^{\mrm{fin}}\big)=-\big(S_{-\infty}-P_{\mfrak{X}}^{\mrm{fin}}\big)
\end{equation*}
of projection operators.
We may decompose $P_{0}$ as $P_{0}=\sum_{x\in\mfrak{X}}(P_{0})_{x}$ according to $\mcal{K}=\bigoplus_{x\in\mfrak{X}}\mcal{K}_{x}$.
Then, we can see that, for each $x\in\mfrak{X}$,
\begin{equation*}
	P_{x}-(P_{0})_{x}=\left(
	\begin{matrix}
		|u^{+}_{x}|^{2}-1 & u^{+}_{-x}u^{-}_{x} \\
		\overline{u^{+}_{-x}u^{-}_{x}} & |u^{-}_{x}|^{2}
	\end{matrix}
	\right)
\end{equation*}
and
\begin{equation*}
	\tr_{\mcal{K}_{x}}\big((P_{x}-(P_{0})_{x})^{2}\big)=\big(1-|u^{+}_{x}|^{2}\big)^{2}+2|u^{+}_{-x}u^{-}_{x}|^{2}+|u^{-}_{x}|^{4}.
\end{equation*}
Hence, we have
\begin{equation*}
\tr_{\mcal{K}}\big((S_{\infty}-P_{0})^{2}\big) =\sum_{x\in\mfrak{X}}\big(\big(1-|u^{+}_{x}|^{2}\big)^{2}+2|u^{+}_{-x}u^{-}_{x}|^{2}+|u^{-}_{x}|^{4}\big).
\end{equation*}
On the other hand, we observe from (\ref{eq:example_wavefunction}) that
\begin{gather*}
	|u^{+}_{x}|^{2} =1+O\big((|\Delta (x)|/\Upsilon (x))^{2}\big), \qquad
	|u^{-}_{x}|^{2} =O\big((|\Delta (x)|/\Upsilon (x))^{2}\big)
\end{gather*}
as $|x|\to\infty$.
Therefore, if $|\Delta (x)|/\Upsilon (x)=o\big(|x|^{-1/2}\big)$ as $|x|\to\infty$, $S_{\infty}-P_{0}=-\big(S_{-\infty}-P_{\mfrak{X}}^{\mrm{fin}}\big)$ is of Hilbert--Schmidt class, which implies, due to Propositions~\ref{prop:finitary_PfPP} and~\ref{prop:finitary_PfPP_generalization}, that the PfPPs $M^{S_{\infty}}$ and $M^{S_{-\infty}}$ are supported on $\Omega^{\circ}$ and $\Omega_{\mfrak{X}}^{\mrm{fin}}$, respectively.
\end{proof}

\section{From measures to states}\label{sect:measre_to_state}
\subsection{Quasi-invariant measures and representations}
This and next subsections are devoted to an exposition of a construction of a state over the gauge-invariant subalgebra of a CAR algebra from a probability measure that was proposed in~\cite{Olshanski2020}.
\subsubsection{Koopman-type construction}
Let $G$ be a countable group of automorphisms of $(\Omega,\Sigma)$. Then, $G$ naturally acts on the collection of measures on $(\Omega,\Sigma)$ by
\begin{equation*}
	{}^{g}M(A):=M\big(g^{-1}(A)\big),\qquad A\in \Sigma,\qquad g\in G.
\end{equation*}
Two measures $M_{1}$ and $M_{2}$ are said to be equivalent if they are absolutely continuous with respect to each other and, in this case, we write $M_{1}\simeq M_{2}$. By means of this notion, we say that a measure $M$ is $G$-quasi-invariant if $M\simeq {}^{g}M$ for arbitrary $g\in G$.

Since the group $G$ acts naturally on the commutative algebra $C(\Omega)$ of continuous functions on $\Omega$, we can consider a semi-direct product of C$^{\ast}$-algebras $C(\Omega)\rtimes G$. Take a $G$-quasi-invariant measure $M$ on $(\Omega,\Sigma)$.
We define a representation of $C(\Omega)\rtimes G$ on $L^{2}(\Omega,M)$ following the Koopman-type construction:
\begin{alignat*}{3}
	&\mcal{T}^{M}(F)h:=Fh, \qquad && F\in C(\Omega),\quad h\in L^{2}(\Omega, M),& \\
	&(\mcal{T}^{M}(g)h)(\omega):=h(g^{-1}(\omega))\phi (\omega,g )^{1/2}, \qquad && g\in G,\quad h\in L^{2}(\Omega,M),\quad \omega\in \Omega,&
\end{alignat*}
where $\phi$ is the 1-cocycle defined by
\begin{equation*}
	\phi (\omega,g):=\frac{{}^{g}M}{M}(\omega),\qquad \omega\in \Omega,\qquad g\in G.
\end{equation*}
Notice that the $G$-quasi-invariance ensures the existence of the Radon--Nikod{\'y}m derivative.

\subsubsection[Arising of the gauge-invariant subalgebra A(K,Gamma)]{Arising of the gauge-invariant subalgebra $\boldsymbol{\mcal{A}(\mcal{K},\Gamma)}$}

Hereafter, we assume that $\Omega=\{0,1\}^{\mfrak{X}}$ and $G=\mfrak{S}=\mfrak{S}(\mfrak{X})$ that consists of finite permutations of $\mfrak{X}$.
The wreath product $\mfrak{S}\wr\mbb{Z}_{2}$ is realized as the semi-direct product $\mfrak{S}\ltimes \mcal{E}$, where $\mcal{E}$ is the commutative algebra generated by $\varepsilon_{x}$, $x\in\mfrak{X}$ with the relations $\varepsilon_{x}^{2}=1$, $x\in\mfrak{X}$. The covariance structure reads
\begin{equation*}
	g\varepsilon_{x}g^{-1}=\varepsilon_{g(x)},\qquad g\in\mfrak{S},\qquad x\in\mfrak{X}.
\end{equation*}

For each $x\in \mfrak{X}$, we define $d_{x}\in C(\Omega)$ by $d_{x}(\omega)=1-2\omega(x)$, $\omega\in\Omega$.
\begin{Proposition}
Let us write $C^{\ast}[\mfrak{S}\wr\mbb{Z}_{2}]$ for the C$^{\ast}$-algebra completion of the group algebra $\mbb{C}[\mfrak{S}\wr\mbb{Z}_{2}]$. We have an isomorphism
\begin{equation*}
	C(\Omega)\rtimes \mfrak{S}\to C^{\ast}[\mfrak{S}\wr\mbb{Z}_{2}],
\end{equation*}
which is characterized by the assignment $d_{x}\mapsto \varepsilon_{x}$, $x\in\mfrak{X}$ and the natural identification of $\mfrak{S}$ in both sides.
\end{Proposition}

Let $M$ be an $\mfrak{S}$-quasi-invariant measure and $\mcal{T}^{M}$ be the associated representation of $C(\Omega)\rtimes\mfrak{S}$ on $L^{2}(\Omega,M)$. Then, due to the above isomorphism, it is regarded as a representation of $C^{\ast}[\mfrak{S}\wr\mbb{Z}_{2}]$.

We introduce a two-sided ideal
\begin{equation*}
	I=\langle (1-s_{x,y})(1-\varepsilon_{x})(1-\varepsilon_{y}), (1-s_{x,y})(1+\varepsilon_{x})(1+\varepsilon_{y})\colon x,y\in \mfrak{X}\rangle
\end{equation*}
of $C^{\ast}[\mfrak{S}\wr\mbb{Z}_{2}]$, where we write $s_{x,y}$ for the transposition of~$x$ and~$y$.
Since it is generated by self-adjoint elements, the quotient $C^{\ast}[\mfrak{S}\wr\mbb{Z}_{2}]/I$ is a C$^{\ast}$-algebra.

\begin{Proposition}
For an $\mfrak{S}$-quasi-invariant measure $M$, the representation $\mcal{T}^{M}$ factors through the quotient $C^{\ast}[\mfrak{S}\wr\mbb{Z}_{2}]/I$.
\end{Proposition}

To define a morphism $C^{\ast}[\mfrak{S}\wr\mbb{Z}_{2}]\to\mcal{A}(\mcal{K},\Gamma)$, we further assume that $\mfrak{X}$ is equipped with a linear order $\le$ so that, as an ordered set, $(\mfrak{X},\le)$ is isomorphic to $\mbb{Z}$ or $\mbb{N}$. In particular, we assume that each interval is a finite set.
We write
\begin{equation*}
	\eta_{x}:=1-2a^{\ast}_{x}a_{x},\qquad x\in\mfrak{X}
\end{equation*}
and
\begin{equation*}
	\eta_{(x,y)}:=\prod_{x<z<y}\eta_{z},\qquad x,y\in\mfrak{X}.
\end{equation*}
Notice that $\eta_{x}$, $x\in\mfrak{X}$ are commutative and the ordering of the product does not matter.
\begin{Proposition}
We defina a morphism $p\colon C^{\ast}[\mfrak{S}\wr\mbb{Z}_{2}]\to\mcal{A}(\mcal{K},\Gamma)$ by
\begin{gather*}
	p(\varepsilon_{x}):=\eta_{x},\qquad x\in\mfrak{X}, \\
p(s_{x,y}) :=\frac{1+\eta_{x}\eta_{y}}{2}+\frac{1+\eta_{x}\eta_{y}+(1-\eta_{x}\eta_{y})\eta_{(x,y)}}{2}(a^{\ast}_{x}a_{y}+a^{\ast}_{y}a_{x}),\qquad x<y.
\end{gather*}
Then, $p$ is a surjection and $\ker p=I$. Hence, in particular, $C^{\ast}[\mfrak{S}\wr\mbb{Z}_{2}]/I\simeq \mcal{A}(\mcal{K},\Gamma)$.
\end{Proposition}

\subsection{Construction of states}
When $M$ is a probability measure on $(\Omega,\Sigma)$, we may define a state $\varphi^{M}$ over $C^{\ast}[\mfrak{S}\wr\mbb{Z}_{2}]/I\simeq \mcal{A}(\mcal{K},\Gamma)$ by
\begin{equation*}
	\varphi^{M}(A):=\big(\mbb{I},\mcal{T}^{M}(A)\mbb{I}\big)_{L^{2}(\Omega, M)},\qquad A\in C^{\ast}[\mfrak{S}\wr\mbb{Z}_{2}]/I\simeq \mcal{A}(\mcal{K},\Gamma),
\end{equation*}
where $\mbb{I}\in L^{2}(\Omega, M)$ is the unit constant function.
When we restrict this state on the subalgebra $C(\Omega)$, we see that
\begin{equation*}
	\varphi^{M}(F)=\int_{\Omega}F(\omega)M({\rm d}\omega),\qquad F\in C(\Omega),
\end{equation*}
which is exactly the expectation value of $F$ under the probability measure~$M$.
Therefore, if $M$ is a PfPP,
\begin{equation*}
	\varphi^{M}\big(a_{x_{1}}^{\ast}\cdots a_{x_{n}}^{\ast}a_{x_{n}}\cdots a_{x_{1}}\big)=\int_{\Omega_{x_{1},\dots, x_{n}}}M({\rm d}\omega)=\pf\big[\mbb{K}^{M}(x_{i},x_{j})\big]_{1\le i,j\le n}.
\end{equation*}
Therefore, we are tempted to expect that the state $\varphi^{M}$ is quasifree, but it is not obvious.

\subsection{Perfectness: Warm-up}
To illustrate an idea of proving Theorem~\ref{thm:perfectness_Schur}, let us start with a simple example.
Let us take linearly independent vectors ${\bf v}=\big\{v_{n}\in\ell^{2}(\mfrak{X})\,|\, n=1,\dots, N\big\}$ and let $K_{{\bf v}}$ be the orthogonal projection to the subspace spanned by ${\bf v}$. Then, the projection $P_{{\bf v}}=(I-JK_{{\bf v}}J)\oplus K_{{\bf v}}$ on $\mcal{K}$ lies in $\mrm{Gr}(\mcal{K},\Gamma)$. Since, in particular, it preserves each component of $\mcal{K}=\ell^{2}(\mfrak{X})\oplus\ell^{2}(\mfrak{X})$, it determines a DPP $M^{P_{{\bf v}}}$ with correlation kernel $K_{{\bf v}}(x,y)=(e_{x},K_{{\bf v}}e_{y})$.

\begin{Proposition}
\label{prop:perfectness_warmup}
Let us expand each vector $v_{n}$ as $v_{n}=\sum_{x\in\mfrak{X}}v_{n}(x)e_{x}$, $n=1,\dots, N$.
If
\begin{equation*}
	\det (v_{i}(x_{j}) )_{1\le i,j\le N}\ge 0,\qquad x_{1}>\cdots >x_{N},
\end{equation*}
then the DPP $M^{P_{{\bf v}}}$ is perfect.
\end{Proposition}
\begin{proof}
It is immediate that $P_{{\bf v}}-P_{0}$ is of Hilbert--Schmidt class as far as $N<\infty$. Therefore, the Bogoliubov automorphism induced from a unitary $V_{{\bf v}}$ such that $P_{{\bf }}=V_{{\bf v}}^{\ast}P_{0}V_{{\bf v}}$ is unitarily implementable on $\big(\pi_{P_{0}},\mcal{F}\big(\ell^{2}(\mfrak{X})\big)\big)$.
Let us denote such a unitary operator by $\mcal{V}_{{\bf v}}$.
Since the representation $\big(\pi_{P_{0}}\circ \alpha_{V},\mcal{F}\big(\ell^{2}(\mfrak{X})\big)\big)$ is the GNS representation of $\varphi_{P_{{\bf v}}}$, we have
\begin{equation*}
	\varphi_{P_{{\bf v}}}(A)=\big(\bm{1}, \mcal{V}^{\ast}_{{\bf v}}\pi_{P_{0}}(A)\mcal{V}_{{\bf v}}\bm{1}\big)_{\mcal{F}(\ell^{2}(\mfrak{X}))},\qquad A\in \mcal{C}(\mcal{K},\Gamma).
\end{equation*}

Let us take an orthonormal basis $\{\phi_{n}\}_{n=1}^{N}$ of the space $\mrm{Span}\{v_{n}\}_{n=1}^{N}$ and expand each of them as $\phi_{n}=\sum_{x\in\mfrak{X}}\phi_{n}(x)e_{x}$.
In the particular case we are considering, the unitary operator $\mcal{V}_{{\bf v}}$ acts on the vacuum vector as \cite{Ruijsenaars1978}
\begin{equation*}
	\mcal{V}_{{\bf v}}\bm{1}=a_{P_{0}\mcal{K}}^{\ast}(\phi_{1})\cdots a_{P_{0}\mcal{K}}^{\ast}(\phi_{N})\bm{1}={\rm e}^{{\rm i}\theta}\sum_{\omega=\{x_{1}>\cdots >x_{N}\}\in\Omega_{N}}\det \left(\phi_{i}(x_{j})\right)_{1\le i,j\le N}e_{\omega}
\end{equation*}
with some constant $\theta\in [0,2\pi)$ independent of $\omega\in\Omega_{N}$.
Now, there exist constants $Z_{N}> 0$ and $\varphi\in [0,2\pi)$ independent of $x_{1},\dots, x_{N}$ such that
\begin{equation*}
{\rm e}^{{\rm i}\varphi}\det \left(\phi_{i}(x_{j})\right)_{1\le i,j\le N}=Z_{N}^{-1/2}\det \left(v_{i}(x_{j})\right)_{1\le i,j\le N}.
\end{equation*}
Note that the right hand side is non-negative from the assumption. Therefore, we have
\begin{equation}\label{eq:expansion_transformed_vacuum}
	\mcal{V}_{{\bf v}}\bm{1}={\rm e}^{{\rm i}(\theta-\varphi)}\sum_{\omega\in \Omega_{N}} M^{P_{{\bf v}}}(\omega)^{1/2}e_{\omega}.
\end{equation}
Here, we can choose $\theta=\varphi$ without loss of generality.

Next, we construct an injective homomorphism $\iota\colon L^{2}\big(\Omega, M^{P_{{\bf v}}}\big)\hookrightarrow \mcal{F}\big(\ell^{2}(\mfrak{X})\big)$ of Hilbert spaces. Noting that $L^{2}\big(\Omega, M^{P_{{\bf v}}}\big)$ is realized as $ \ell^{2}\big(\mrm{supp}M^{P_{{\bf v}}}, M^{P_{{\bf v}}}\big)$, we define it as
\begin{equation*}
	\iota\colon \ \delta_{\omega}\mapsto M^{P_{{\bf v}}}(\omega)^{1/2}e_{\omega},
\end{equation*}
where $\delta_{\omega}$, $\omega\in\Omega$ is the unit function supported at $\omega$. Then it is obviously an isometry. In particular, the unit constant function $\mbb{I}$ is mapped to $\mcal{V}_{{\bf v}}\bm{1}$. It remains to show that the homomorphism $\iota$ intertwines the representations $\mcal{T}^{M^{P_{{\bf v}}}}$ and $\pi_{P_{0}}$:
\begin{equation*}
	\iota\circ \mcal{T}^{M^{P_{{\bf v}}}}(A)= \pi_{P_{0}}(A)\circ \iota,\qquad A\in \mcal{A}(\mcal{K},\Gamma).
\end{equation*}
In fact, it implies that $\varphi^{M^{P_{{\bf v}}}}=\varphi_{P_{{\bf v}}}|_{\mcal{A}(\mcal{K},\Gamma)}$.

Let us investigate the actions $\pi_{P_{0}}(p(\varepsilon_{x}))$, $x\in\mfrak{X}$ and $\pi_{P_{0}}(p(s_{x,y}))$, $x<y$. It is immediate that
$\pi_{P_{0}}(p(\varepsilon_{x}))e_{\omega}=d_{x}(\omega) e_{\omega}$,
which implies that
\begin{equation*}
	\pi_{P_{0}}(F)e_{\omega}=F(\omega) e_{\omega},\qquad \omega\in \Omega^{\circ},\qquad F\in C(\Omega).
\end{equation*}
To investigate the action of $\pi_{P_{0}}(p(s_{x,y}))$, $x<y$, we take $\omega\in \Omega^{\circ}$ arbitrarily and set $\pr{\omega}=s_{x,y}(\omega)$.
We consider the following distinct cases:
\begin{enumerate}\itemsep=0pt\samepage
\item 	When $\omega$ does not contain neither $x$ nor $y$, we have $\pr{\omega}=\omega$.
We have
				\begin{equation*}
					\pi_{P_{0}}(a_{x}^{\ast}a_{y}+a_{y}^{\ast}a_{x})e_{\omega}=0.
				\end{equation*}
				Hence,
				\begin{equation*}
\pi_{P_{0}}(p(s_{x,y}))e_{\omega}=\frac{1}{2}\pi_{P_{0}}(1+\eta_{x}\eta_{y})e_{\omega}=e_{\omega}=e_{\pr{\omega}}.
				\end{equation*}
\item 	When $\omega$ contains both $x$ and $y$, we again have $\pr{\omega}=\omega$. In this case,
				\begin{equation*}
\pi_{P_{0}}(a_{x}^{\ast}a_{y}+a_{y}^{\ast}a_{x})e_{\omega}=-\pi_{P_{0}}(a_{y}a_{x}^{\ast}+a_{x}a_{y}^{\ast})e_{\omega}=0.
				\end{equation*}
				Hence,
				\begin{equation*}
\pi_{P_{0}}(p(s_{x,y}))e_{\omega}=\frac{1}{2}\pi_{P_{0}}(1+\eta_{x}\eta_{y})e_{\omega}=e_{\omega}=e_{\pr{\omega}}.
				\end{equation*}
		\item 	When $x\in\omega$ and $y\not\in\omega$, $\pr{\omega}=\omega\backslash \{x\}\cup \{y\}$.
				Suppose that $e_{\omega}$ has the form
				\begin{equation*}
e_{\omega}=\underbrace{\cdots}_{\ell_{1}} \wedge \hat{e}_{y} \wedge\underbrace{\cdots}_{\ell_{2}} \wedge e_{x}\wedge\cdots,
				\end{equation*}
				where $\hat{e}_{y}$ means that $e_{y}$ is removed from the corresponding position.
				Then, we have
				\begin{gather*}
\pi_{P_{0}}(a^{\ast}_{x}a_{y}+a_{y}^{\ast}a_{x})e_{\omega}=(-1)^{\ell_{1}+\ell_{2}}(-1)^{\ell_{1}}e_{\pr{\omega}}=(-1)^{\ell_{2}}e_{\pr{\omega}} \\
					\pi_{P_{0}}(1+\eta_{x}\eta_{y})e_{\omega}=0, \qquad
					\pi_{P_{0}}(1-\eta_{x}\eta_{y})e_{\omega}=2e_{\omega}, \qquad
					\pi_{P_{0}}(\eta_{(x,y)})e_{\omega}=(-1)^{\ell_{2}}e_{\omega}.
				\end{gather*}
				These properties verify that
				\begin{equation*}
					\pi_{P_{0}}(p(s_{x,y}))e_{\omega}=e_{\pr{\omega}}.
				\end{equation*}
		\item 	When $x\not\in\omega$ and $y\in\omega$, the same property is verified in a similar argument.
		\end{enumerate}
		To conclude, we have $\pi_{P_{0}}(p(g))e_{\omega}=e_{g(\omega)}$, $g\in\mfrak{S}$.

For $F\in C(\Omega)$, it is obvious that
\begin{equation*}
	\iota\circ \mcal{T}^{M^{P_{{\bf v}}}}(F)\delta_{\omega}=M^{P_{{\bf v}}}(\omega)^{1/2}F(\omega)e_{\omega}=\pi_{P_{0}}(F)\circ \iota \delta_{\omega},\qquad \omega\in\Omega_{N}.
\end{equation*}
For $g\in\mfrak{S}$, we have
\begin{equation*}
	\iota\circ \mcal{T}^{M^{P_{{\bf v}}}}(p(g))\delta_{\omega}=\iota \delta_{g(\omega)}\left(\frac{{}^{g}M^{P_{{\bf v}}}}{M^{P_{{\bf v}}}}(g(\omega))\right)^{1/2}=M^{P_{{\bf v}}}(\omega)^{1/2}e_{g(\omega)},\qquad \omega\in\Omega_{N},
\end{equation*}
while, on the other hand,
\begin{equation*}
	\pi_{P_{0}}(p(g))\circ \iota \delta_{\omega}=\pi_{P_{0}}(p(g)) M^{P_{{\bf v}}}(\omega)^{1/2}e_{\omega}=M^{P_{{\bf v}}}(\omega)^{1/2}e_{g(\omega)},\qquad \omega\in \Omega_{N}.
\end{equation*}
Therefore, we can conclude that $\iota$ intertwines representations $\mcal{T}^{M^{P_{{\bf v}}}}$ and $\pi_{P_{0}}$, and the proof is complete.
\end{proof}

Notice that, in the above arguments, it is essential that the coefficients in the expansion~(\ref{eq:expansion_transformed_vacuum}) are non-negative; in general, the squared absolute value of each coefficient gives a weight of the probability measure.

\begin{Example}
Suppose that $\mfrak{X}\subset\mbb{R}$ with the induced order.
For a weight function $W(x)$ such that
\begin{equation*}
	\sum_{x\in\mfrak{X}}x^{2N}W(x)^{1/2}<\infty,
\end{equation*}
we take $v_{n}=x^{N-n}W(x)^{1/2}$, $n=1,\dots, N$. The corresponding DPP $M^{P_{{\bf v}}}$ is a discrete orthogonal polynomial ensemble~\cite{Konig2005}.
It is immediate that
\begin{equation*}
	\det \big(x^{N-i}_{j}W(x_{j})^{1/2}\big)_{1\le i,j\le N}=\prod_{i=1}^{N}W(x_{i})^{1/2}\prod_{1\le i<j\le N}(x_{i}-x_{j})\ge 0,\qquad x_{1}>\cdots >x_{N}.
\end{equation*}
Therefore, $M^{P_{{\bf v}}}$ is perfect as was shown in~\cite{Olshanski2020} by directly estimating the correlation kernel.
\end{Example}

\subsection{Schur measures}\label{subsect:Schur_measures}
\subsubsection{Schur functions and positive specialization}
Let $\Lambda_{n}=\mbb{C}[x_{1},\dots, x_{n}]^{\mfrak{S}_{n}}$ be the ring of symmetric polynomials of $n$ variables. We write $\Lambda=\varprojlim_{n}\Lambda_{n}$ for the projective limit in the category of graded rings and call it the ring of symmetric functions. Note that an object like $\prod_{i\ge 1}(1+x_{i})$ is not, counter-intuitively, a symmetric function. We set $p_{n}=\sum_{i\ge 1}x_{i}^{n}$ and call it the $n$-th power-sum symmetric function. Then, the power-sum symmetric functions freely generate $\Lambda$ so that $\Lambda=\mbb{C}[p_{1},p_{2},\dots]$. The ring of symmetric functions $\Lambda$ admits a distinguished basis $\{s_{\lambda}\,|\,\lambda\in\mbb{Y}\}$ constituted with the Schur functions, which are characterized in several manners (see~\cite{Macdonald1999}).

An algebraic homomorphism $\tau\colon \Lambda\to\mbb{C}$ is said to be Schur-positive if $\tau(s_{\lambda})\ge 0$ for all $\lambda\in\mbb{Y}$. It is a classical result~\cite{AissenSchoenbergWhitney1952, Edrei1952, Thoma1964} that Schur-positive specializations are parametrized by the set~$\mbb{T}$ in the way that $\rho=(\alpha;\beta)$ gives a Schur-positive specialization $\tau_{\rho}$ defined by $\tau_{\rho}(p_{1})=1$ and
\begin{equation*}
	\tau_{\rho}(p_{n})=\sum_{i\ge 1}\alpha_{i}^{n}+(-1)^{n-1}\sum_{i\ge 1}\beta_{i}^{n},\qquad n\ge 2.
\end{equation*}
For a symmetric function $F\in \Lambda$, we often write its image under $\tau_{\rho}$ as $F(\rho)$ instead of $\tau_{\rho}(F)$.

\subsubsection{Free fermion description} Here, we consider the case when $\mfrak{X}=\mbb{Z}+\frac{1}{2}$. Let us write $P^{\mrm{S}}_{0}$ for the orthogonal projection onto $\mcal{K}^{+}_{\mbb{Z}_{\ge 0}+\frac{1}{2}}\oplus \mcal{K}^{-}_{\mbb{Z}_{\le 0}-\frac{1}{2}}$. Then, it is obvious that $P^{\mrm{S}}_{0}\in\mrm{Gr}(\mcal{K},\Gamma)$. It is standard to realize the corresponding Fock space $\mcal{F}(P^{\mrm{S}}_{0}\mcal{K})$ as a space of infinite wedges. Let $\Omega^{\mrm{S}}$ be the collection of $\omega\in\Omega$ such that
\begin{equation*}
	\omega_{+}:=\omega\cap \left(\mbb{Z}_{\ge 0}+\frac{1}{2}\right)\qquad \mbox{and}\qquad \omega_{-}:=\left(\mbb{Z}_{\le 0}-\frac{1}{2}\right)\bigg\backslash \omega
\end{equation*}
are finite.
Equivalently, each element $\omega\in\Omega^{\mrm{S}}$ is a collection $\{x_{1}>x_{2}>\cdots\}$ such that $x_{1}<\infty$ and $x_{j+1}=x_{j}-1$ for all sufficiently large $j$.
Then, the Fock space $\mcal{F}\big(P^{\mrm{S}}_{0}\mcal{K}\big)$ admits a complete orthonormal system $\big\{e_{\omega}\,|\,\omega\in \Omega^{\mrm{S}}\big\}$, where
\begin{equation*}
	e_{\omega}=e_{x_{1}}\wedge e_{x_{2}}\wedge\cdots,\qquad \omega=\{x_{1}>x_{2}>\cdots\}\in\Omega^{\mrm{S}}.
\end{equation*}
The action of the CAR algebra $\mcal{C}(\mcal{K},\Gamma)$ has a natural description:
\begin{equation*}
	\pi_{P_{0}^{\mrm{S}}}(a_{x}^{\ast}) \eta:=e_{x}\wedge \eta,\qquad x\in\mfrak{X},\qquad \eta\in \mcal{F}\big(P_{0}^{\mrm{S}}\mcal{K}\big),
\end{equation*}
and $\pi_{P_{0}^{\mrm{S}}}(a_{x})$ acts as its adjoint operator. The cyclic vector $\bm{1}$ is identified with
\begin{equation*}
	\bm{1}=e_{-1/2}\wedge e_{-3/2}\wedge\cdots.
\end{equation*}

Under the embedding, $\mbb{M}\colon \mbb{Y}\hookrightarrow \Omega$, the image is included in $\Omega^{\mrm{S}}$. Strictly speaking, the image is isomorphic to a subset of $\Omega^{\mrm{S}}$ consisting of $\omega$ such that $\#\omega_{+}=\#\omega_{-}$. Under the inclusion $\mbb{M}$, the empty partition $\empty$ is mapped to the cyclic vector $\bm{1}$.

We set $\mcal{D}^{\mrm{S}}=\mrm{Span}\big\{e_{\omega}\,|\,\omega\in\Omega^{\mrm{S}}\big\}$ for a dense subspace of the Fock space $\mcal{F}\big(P^{\mrm{S}}_{0}\mcal{K}\big)$
Observe that, as operators on $\mcal{F}(P^{\mrm{S}}_{0}\mcal{K})$,
\begin{equation*}
	h_{n}:=\sum_{x\in\mfrak{X}}\pi_{P_{0}^{\mrm{S}}}(a^{\ast}_{x-n}a_{x}),\qquad n\in\mbb{Z}\backslash\{0\}
\end{equation*}
make sense with a dense domain $\mcal{D}^{\mrm{S}}$ and exhibit the Heisenberg commutation relations $[h_{m},h_{n}]=m\delta_{m+n,0}$, $m,n\in\mbb{Z}\backslash\{0\}$.
Notice that $h_{n}^{\ast}=h_{-n}$, $n\in\mbb{Z}\backslash\{0\}$.
For each $\rho\in\mbb{T}^{\circ}$, we introduce operators on $\mcal{F}\big(P^{\mrm{S}}_{0}\mcal{K}\big)$ by
\begin{equation*}
	\Xi _{\pm}(\rho)=\exp\left(\sum_{n=1}^{\infty}\frac{p_{n}(\rho)}{n}h_{\pm n}\right).
\end{equation*}

\begin{Proposition}
For $\rho\in\mbb{T}^{\circ}$, the operators $\Xi_{\pm}(\rho)$ are well-defined with a dense domain $\mcal{D}^{\mrm{S}}$.
\end{Proposition}
\begin{proof}
First, let us verify that $\Xi_{+}(\rho)e_{\omega}\in\mcal{D}^{\mrm{S}}$ is at most a finite linear combination of $e_{\omega}$, $\omega\in\Omega^{\mrm{S}}$.
To this aim, we introduce the following energy operator:
\begin{equation*}
	H=\sum_{x>0}x\pi_{P_{0}^{\mrm{S}}}(a_{x}^{\ast}a_{x})+\sum_{x<0}(-x)\pi_{P_{0}^{\mrm{S}}}(a_{x}a_{x}^{\ast}).
\end{equation*}
Then the complete orthonormal system $\big\{e_{\omega}\,|\,\omega\in\Omega^{\mrm{S}}\big\}$ diagonalizes $H$. In fact, for $\omega\in \Omega^{\mrm{S}}$, we have
\begin{equation*}
	He_{\omega}=\left(\sum_{x\in \omega_{+}}x+\sum_{x\in\omega_{-}}(-x)\right)e_{\omega}.
\end{equation*}
In particular, we see that the spectrum of $H$ coincides with $\frac{1}{2}\mbb{Z}_{\ge 0}$.
We can also show, by direct computation, that $[H,h_{n}]=-nh_{n}$, $n\in\mbb{Z}$, which implies that each operator $h_{n}$ lowers the eigenvalue of $H$ by $n$. Hence, for any $\omega\in\Omega^{\mrm{S}}$, $h_{n_{1}}\cdots h_{n_{k}}e_{\omega}=0$ whenever $n_{1}+\cdots +n_{k}$ is sufficiently large. Therefore, $\Xi_{+}(\rho)e_{\omega}\in\mcal{D}^{\mrm{S}}$.

Next, we show that $\|\Xi_{-}(\rho)\bm{1}\|<\infty$ if $\rho\in\mbb{T}^{\circ}$. From the properties $\Xi_{\pm}(\rho)^{\ast}=\Xi_{\mp}(\rho)$ and $\Xi_{+}(\rho)\bm{1}=0$, we can see that the squared norm $\|\Xi_{-}(\rho)\bm{1}\|^{2}$ appears in a commutation relation as
\begin{equation*}
	\Xi_{+}(\rho)\Xi_{-}(\rho)=\|\Xi_{-}(\rho)\bm{1}\|^{2}\Xi_{-}(\rho)\Xi_{+}(\rho).
\end{equation*}
Hence we have $\|\Xi_{-}(\rho)\bm{1}\|^{2}=\exp\left(\sum_{n=1}^{\infty}\frac{p_{n}(\rho)^{2}}{n}\right)$.
If $\rho$ is trivial; $\alpha_{1}=\beta_{1}=0$, then $\|\Xi_{-}(\rho)\bm{1}\|=e<\infty$.
Let us assume that $\alpha_{1}>0$. For any $n\in\mbb{N}$, we find the following estimation
\begin{align*}
	&\left(\sum_{j\ge 1}\alpha_{j}^{n}+(-1)^{n-1}\sum_{j\ge 1}\beta_{j}^{n}\right)^{2}
	\le \alpha_{1}^{2n}+\left(\sum_{j\ge 2}\alpha_{j}+\sum_{j\ge 1}\beta_{j}\right)^{2n}+2\alpha_{1}^{n}\left(\sum_{j\ge 2}\alpha_{j}+\sum_{j\ge 1}\beta_{j}\right)^{n},
\end{align*}
where, by assumption, $\alpha_{1}$ and $\sum_{j\ge 2}\alpha_{j}+\sum_{j\ge 1}\beta_{j}$ lie in $(0,1)$, and so does their product.
Hence, we have
\begin{align*}
	\|\Xi_{-}(\rho)\bm{1}\|^{2}\le \frac{\exp\left(1-\left(\sum_{j\ge 1}\alpha_{j}+\sum_{j\ge 1}\beta_{j}\right)^{2}\right)}{\left(1-\alpha_{1}^{2}\right)\left(1-\left(\sum_{j\ge 2}\alpha_{j}+\sum_{j\ge 1}\beta_{j}\right)^{2}\right)\left(1-\alpha_{1}\left(\sum_{j\ge 2}\alpha_{j}+\sum_{j\ge 1}\beta_{j}\right)\right)^{2}},
\end{align*}
which is finite by assumption. If $\alpha_{1}=0$, the above estimation works by using $\beta_{1}$ instead of $\alpha_{1}$.

Finally, for general $\omega\in \Omega^{\mrm{S}}$, we see that
\begin{equation*}
	\|\Xi_{-}(\rho)e_{\omega}\|^{2}=\|\Xi_{-}(\rho)\bm{1}\|^{2}\|\Xi_{+}(\rho)e_{\omega}\|^{2}
\end{equation*}
is finite.
\end{proof}

Significantly, we have (e.g., \cite[Appendix~A]{Okounkov2001})
\begin{equation}\label{eq:generating_Schur}
	\Xi_{-}(\rho)\bm{1}=\sum_{\lambda\in\mbb{Y}}s_{\lambda}(\rho)e_{\mbb{M}(\lambda)}.
\end{equation}

We define a state $\varphi_{s(\rho)}$, $\rho\in\mbb{T}^{\circ}$ over $\mcal{C}(\mcal{K},\Gamma)$ by
\begin{align*}
\varphi_{s(\rho)}(A):= \frac{\big(\Xi_{-}(\rho)\bm{1},\pi_{P^{\mrm{S}}_{0}}(A)\Xi_{-}(\rho)\bm{1}\big)_{\mcal{F}(P^{\mrm{S}}_{0} \mcal{K})}}{\|\Xi_{-}(\rho)\bm{1}\|^{2}},\qquad A\in \mcal{C}(\mcal{K},\Gamma).
\end{align*}

\begin{Proposition}\label{prop:Schur_measure_quasifree}
The state $\varphi_{s(\rho)}$ is a quasi-free state and
\begin{equation*}
	\varphi_{s(\rho)}(F)=\int_{\Omega}f(\omega)M_{s(\rho)}({\rm d}\omega),\qquad F\in C(\Omega).
\end{equation*}
\end{Proposition}
\begin{proof}When we set
\begin{equation*}
	\widetilde{\Xi}_{\pm}(\rho):=\exp\left(-\sum_{n=1}^{\infty}\frac{p_{n}(\rho)}{n}h_{\pm n}\right),
\end{equation*}
it is immediate that they are defined on a dense domain $\mcal{D}^{\mrm{S}}$. Moreover, we have $\Xi_{\pm}(\rho)\widetilde{\Xi}_{\pm}(\rho)=\widetilde{\Xi}_{\pm}(\rho)\Xi_{\pm}(\rho)=1$ on $\mcal{D}^{\mrm{S}}$. Hence, it is verified that $\widetilde{\Xi}_{\pm}(\rho)=\Xi_{\pm}(\rho)^{-1}$.
This properties allows us to have
\begin{equation*}
	\varphi_{s(\rho)}(A)=\big(\bm{1},\pi_{P_{0}^{\mrm{S}}(\rho)}(A)\bm{1}\big)_{\mcal{F}(P^{\mrm{S}}_{0}\mcal{K})},\qquad A\in \mcal{C}(\mcal{K},\Gamma),
\end{equation*}
where we set
\begin{equation*}
	\pi_{P_{0}^{\mrm{S}}(\rho)}(A)=\Xi_{+}(\rho)\Xi_{-}(\rho)^{-1}\pi_{P^{\mrm{S}}_{0}}(A)\Xi_{-}(\rho)\Xi_{+}(\rho)^{-1},\qquad A\in \mcal{C}(\mcal{K},\Gamma).
\end{equation*}
It is obvious that the assignment $A\mapsto \pi_{P^{\mrm{S}}(\rho)}(A)$ gives a representation. Therefore, computation of $\varphi_{s(\rho)}$ admits Wick's formula, implying that it is a quasi-free state.
\end{proof}

\subsubsection{Proof of Theorem \ref{thm:perfectness_Schur}}
The property (\ref{eq:generating_Schur}) implies that
\begin{equation*}
	\frac{1}{\|\Xi_{-}(\rho)\bm{1}\|}\Xi_{-}(\rho)\bm{1}=\sum_{\omega\in \Omega^{\mrm{S}}}M_{s(\rho)}(\omega)^{1/2}e_{\omega},
\end{equation*}
where the coefficients are all non-negative; recall that this was a prominent observation in the proof of Proposition \ref{prop:perfectness_warmup}.

Now, it suffices to show that the embedding defined by
\begin{equation*}
	\iota \colon \ L^{2}(\Omega,M_{s(\rho)})\hookrightarrow \mcal{F}(P^{\mrm{S}}_{0}\mcal{K});\qquad \delta_{\omega}\mapsto M_{s(\rho)}(\omega)^{1/2}e_{\omega}
\end{equation*}
intertwines the representations $\mcal{T}^{M_{s(\rho)}}$ and $\pi_{P^{\mrm{S}}_{0}}$, which can be checked in the same manner as in the proof of Proposition \ref{prop:perfectness_warmup}.

\section{Conditional measures}\label{sect:conditional_measures}
\subsection{Proof of Lemma \ref{lem:conditional_projection}}
First, we notice that the subspace $P\mcal{K}_{X,\pr{X}}$ consists of vectors $v\in \mcal{K}_{\mfrak{X}\backslash(X\sqcup\pr{X})}$ such that there exist $a\in \mcal{K}_{X}^{+}$, $\pr{a}\in \mcal{K}_{\pr{X}}^{-}$ and $v+a+\pr{a}\in P\mcal{K}$. Due to the property $\Gamma \mcal{K}_{X}^{\pm}=\mcal{K}_{X}^{\mp}$, we can see that
\begin{equation*}
	\Gamma_{\mfrak{X}\backslash(X\sqcup\pr{X})} P\mcal{K}_{X,\pr{X}}=\big(\overline{P}\mcal{K}+\big(\mcal{K}_{X}^{-}\oplus \mcal{K}_{\pr{X}}^{+}\big)\big)\cap \mcal{K}_{\mfrak{X}\backslash (X\cup \pr{X})}=\overline{P}\mcal{K}_{\pr{X},X}.
\end{equation*}
We can also see that $ (P\mcal{K}_{X,\pr{X}})^{\perp}=\overline{P}\mcal{K}_{\pr{X},X}$. In fact, for $u\in P\mcal{K}_{X,\pr{X}}$ and $v\in \overline{P}\mcal{K}_{\pr{X},X}$, we can take $a\in \mcal{K}_{X}^{+}$, $\pr{a}\in \mcal{K}_{\pr{X}}^{-}$, $b\in \mcal{K}_{X}^{-}$, and $\pr{b}\in \mcal{K}_{\pr{X}}^{+}$ such that
\begin{gather*}
	u+a+\pr{a} \in P\mcal{K}, \qquad v+b+\pr{b}\in \overline{P}\mcal{K}.
\end{gather*}
Now, since $a$, $\pr{a}$, $b$, $\pr{b}$ are mutually orthogonal and orthogonal to $u$ and $v$, we have
\begin{equation*}
	(u,v)_{\mcal{K}_{\mfrak{X}\backslash (X\cup\pr{X})}}=(u+a+\pr{a},v+b+\pr{b})_{\mcal{K}}=0,
\end{equation*}
which implies $(P\mcal{K}_{X,\pr{X}})^{\perp}=\overline{P}\mcal{K}_{\pr{X},X}$.
Therefore, $\overline{P}_{X,\pr{X}}=1-P_{X,\pr{X}}$ holds.

\subsection{Conditioning on the CAR algebra}\label{subsect:conditioning_CAR_alg}
For finite disjoint subsets $X, \pr{X}\subset\mfrak{X}$, the characteristic function on the corresponding cylinder set $C(X,\pr{X})$ is identified as
\begin{equation*}
	\chi_{C(X,\pr{X})}=\prod_{x\in X}a_{x}^{\ast}a_{x} \prod_{x\in \pr{X}} a_{x}a_{x}^{\ast}\in C(\Omega)\subset \mcal{C}(\mcal{K},\Gamma).
\end{equation*}
In general, given a state $\varphi$ over $\mcal{C}(\mcal{K},\Gamma)$ its conditioning over $C(X,\pr{X})$ is a state defined by (see, e.g.,~\cite{RedeiSummers2006})
\begin{equation*}
	\varphi (A|\chi_{C(X,\pr{X})})=\frac{\varphi (\chi_{C(X,\pr{X})}A\chi_{C(X,\pr{X})})}{\varphi (\chi_{C(X,\pr{X})})},\qquad A\in \mcal{C}(\mcal{K},\Gamma).
\end{equation*}
According to the decomposition $\mfrak{X}=X\sqcup \pr{X}\sqcup \mfrak{X}\backslash (X\sqcup\pr{X})$, we have a decomposition
\begin{equation*}
	\mcal{C}(\mcal{K},\Gamma)=\mcal{C}(\mcal{K}_{X\sqcup \pr{X}},\Gamma_{X\sqcup\pr{X}})\otimes \mcal{C}(\mcal{K}_{\mfrak{X}\backslash (X\sqcup\pr{X})}),
\end{equation*}
which enables us to regard $C(\mcal{K}_{\mfrak{X}\backslash (X\sqcup\pr{X})},\Gamma_{\mfrak{X}\backslash (X\sqcup\pr{X})})$ as a subalgebra of $\mcal{C}(\mcal{K},\Gamma)$ that coincides with the subalgebra realized as $\chi_{C(X,\pr{X})}\mcal{C}(\mcal{K},\Gamma)\chi_{C(X,\pr{X})}$.
Therefore, the conditional state $\varphi (\cdot|\chi_{C(X,\pr{X})})$ is regarded as a state over $C(\mcal{K}_{\mfrak{X}\backslash (X\sqcup\pr{X})},\Gamma_{\mfrak{X}\backslash (X\sqcup\pr{X})})$ and computed as
\begin{equation*}
	\varphi (A|\chi_{C(X,\pr{X})})=\frac{\varphi (A\chi_{C(X,\pr{X})})}{\varphi (\chi_{C(X,\pr{X})})},\qquad A\in C(\mcal{K}_{\mfrak{X}\backslash (X\sqcup\pr{X})},\Gamma_{\mfrak{X}\backslash (X\sqcup\pr{X})}).
\end{equation*}
In particular, in the case when $\varphi=\varphi_{S}$ is a quasi-free state associated with $S\in \mcal{Q}(\mcal{K},\Gamma)$, we have
\begin{equation*}
	\varphi_{S}(F|\chi_{C(X,\pr{X})})=\int_{\Omega (\mfrak{X}\backslash (X\sqcup\pr{X}))}F(\omega)M^{S}_{X,\pr{X}}({\rm d}\omega),\qquad F\in C(\Omega (\mfrak{X}\backslash (X\sqcup\pr{X}))).
\end{equation*}

\subsection{Some observations}
We see that obtaining $P_{X,\pr{X}}$ from a projection operator $P\in\mrm{Gr}(\mcal{K},\Gamma)$ is decomposed into fundamental steps.
\begin{Lemma}
\label{lem:rest_decomp} We have the following decomposition properties.
\begin{enumerate}\itemsep=0pt
\item[$1.$] 	When $X, \pr{X}\subset\mfrak{X}$ are disjoint finite subsets, we have
		\begin{equation}		\label{eq:rest_proj_decomp}
			P_{X,\pr{X}}=(P_{\varnothing,\pr{X}})_{X,\varnothing}=(P_{X,\varnothing})_{\varnothing, \pr{X}}.
		\end{equation}
\item[$2.$] 	When $X=X_{1}\sqcup X_{2}\subset \mfrak{X}$ is a disjoint union of finite subsets, then
		\begin{equation*}
			P_{X,\varnothing}=(P_{X_{1},\varnothing})_{X_{2},\varnothing}.
		\end{equation*}
\item[$3.$] 	When $\pr{X}=\pr{X}_{1}\sqcup \pr{X}_{2}\subset\mfrak{X}$ is a disjoint union of finite subsets, then
		\begin{equation*}
			P_{\varnothing,\pr{X}}=(P_{\varnothing,\pr{X}_{1}})_{\varnothing,\pr{X}_{2}}.
		\end{equation*}
\end{enumerate}
\end{Lemma}
\begin{proof}We only prove (\ref{eq:rest_proj_decomp}) since the other two follow from arguments of the same type.
We have noticed that $P\mcal{K}$ consists of vectors $u\in \mcal{K}_{\mfrak{X}\backslash (X\sqcup\pr{X})}$ such that there exist $a\in \mcal{K}_{X}^{+}$, $\pr{a}\in \mcal{K}_{\pr{X}}^{-}$ and $u+a+\pr{a}\in P\mcal{K}$. This exactly means that
\begin{equation*}
	u+a\in \{v\in \mcal{K}_{\mfrak{X}\backslash \pr{X}}\,|\, \exists\, \pr{a}\in \mcal{K}_{\pr{X}},\, v+\pr{a}\in P\mcal{K}\},
\end{equation*}
while the latter set is just $P\mcal{K}_{\varnothing,\pr{X}}$. Therefore, $P\mcal{K}_{X,\pr{X}}=(P\mcal{K}_{\varnothing,\pr{X}})_{X,\varnothing}$. The other relation $P\mcal{K}_{X,\pr{X}}=(P\mcal{K}_{X,\varnothing})_{\varnothing,\pr{X}}$ is also shown. Consequently, we reach~(\ref{eq:rest_proj_decomp}).
\end{proof}

We also notice that the regularity is inherited along decomposition.
\begin{Lemma}\label{lem:regularity_decomp} We have the following properties regarding the regularity.
\begin{enumerate}\itemsep=0pt
\item[$1.$] 	For disjoint finite subsets $X,\pr{X}\subset\mfrak{X}$, a projection operator $P\in\mrm{Gr}(\mcal{K},\Gamma)$ is $(X,\pr{X})$-regular if and only if one of the following equivalent conditions holds:
		\begin{enumerate}\itemsep=0pt
		\item[$(a)$] 	$P$ is $(X,\varnothing)$-regular and $P_{X,\varnothing}$ is $(\varnothing,\pr{X})$-regular.
		\item[$(b)$] 	$P$ is $(\varnothing,\pr{X})$-regular and $P_{\varnothing,\pr{X}}$ is $(X,\varnothing)$-regular.
		\end{enumerate}
\item[$2.$] 	When $X=X_{1}\sqcup X_{2}$ is a disjoint union of finite subsets, a projection operator $P\in\mrm{Gr}(\mcal{K},\Gamma)$ is $(X,\varnothing)$-regular if and only if $P$ is $(X_{1},\varnothing)$-regular and $P_{X_{1},\varnothing}$ is $(X_{2},\varnothing)$-regular.
\item[$3.$] 	When $\pr{X}=\pr{X}_{1}\sqcup \pr{X}_{2}$ is a disjoint union of finite subsets, a projection operator $P\in\mrm{Gr}(\mcal{K},\Gamma)$ is $(\varnothing,\pr{X})$-regular if and only if $P$ is $(\varnothing,\pr{X}_{1})$-regular and $P_{\varnothing,\pr{X}_{1}}$ is $(\varnothing,X_{2})$-regular.
\end{enumerate}
\end{Lemma}
\begin{proof}
We only prove the equivalence to the item (a) in (1) since the others are shown in similar arguments.
Suppose that $P\in \mrm{Gr}(\mcal{K},\Gamma)$ is $(X,\pr{X})$-regular. Then, it is obvious that $P$ is $(X,\varnothing)$-regular. Now, $P\mcal{K}_{X,\varnothing}$ consists of vectors $u\in\mcal{K}_{\mfrak{X}\backslash X}$ such that there exists $a\in \mcal{K}_{X}^{+}$ and $u+a\in P\mcal{K}$. Let us take a vector $u_{0}\in P\mcal{K}_{X,\varnothing}\cap \mcal{K}_{\pr{X}}^{-}$. Then, there exists $a_{0}\in \mcal{K}_{X}^{+}$ such that $u_{0}+a_{0}\in P\mcal{K}$, while we also have $u_{0}+a_{0}\in \mcal{K}_{X}^{+}\oplus \mcal{K}_{\pr{X}}^{-}$, which implies that $u_{0}+a_{0}\in P\mcal{K}\cap \mcal{K}_{X}^{+}\oplus \mcal{K}_{\pr{X}}^{-}=\{0\}$ by assumption. Therefore, we must have $u_{0}=0$ and see that $P\mcal{K}_{X,\varnothing}$ is $(\varnothing,\pr{X})$-regular.

Conversely, let us suppose that $P$ is $(X,\varnothing)$-regular and $P_{X,\varnothing}$ is $(\varnothing,\pr{X})$-regular. Take a~nonzero vector $u_{1}\in P\mcal{K}\cap \mcal{K}_{X}^{+}\oplus \mcal{K}_{\pr{X}}^{-}$. Since $P$ is $(X,\varnothing)$-regular, when we write $u_{1}=a_{1}+\pr{a}_{1}$, $a_{1}\in \mcal{K}_{X}^{+}$, $\pr{a}\in \mcal{K}_{\pr{X}}^{-}$, we must have $\pr{a}_{1}\neq 0$. On the other hand, by definition, $\pr{a}_{1}\in P\mcal{K}_{X,\varnothing}$, which implies that $\pr{a}_{1}\in P\mcal{K}_{X,\varnothing}\cap \mcal{K}_{\pr{X}}^{-}$ contradicting the assumption that $P_{X,\varnothing}$ is $(\varnothing,\pr{X})$-regular. Therefore, $P$ is $(X,\pr{X})$-regular.
\end{proof}

Lemmas \ref{lem:rest_decomp} and \ref{lem:regularity_decomp} verify that it suffices to prove Theorem \ref{thm:conditional_projection} in the case when $(X,\pr{X})=(\{x\},\varnothing)$ and $(\varnothing,\{x\})$ with $x\in\mfrak{X}$.
Let us introduce subspaces
\begin{gather*}
	\mcal{R}_{\pm}(x;\mcal{K}):=\big\{u\in \mcal{K}\,|\, \big(\mcal{K}^{\pm}_{\{x\}},u\big)_{\mcal{K}}=0\big\},\qquad x\in \mfrak{X},
\end{gather*}
and set $\mcal{R}_{\pm}(x;P\mcal{K})=P\mcal{K}\cap \mcal{R}_{\pm}(x;\mcal{K})$.
When we write an element $u\in P\mcal{K}$ as $u=(f,g)$ according to the direct sum decomposition $\mcal{K}=\ell^{2}(\mfrak{X})\oplus \ell^{2}(\mfrak{X})$, these are just
\begin{gather*}
	\mcal{R}_{+}(x;P\mcal{K}) =\{(f,g)\in P\mcal{K}\,|\,f(x)=0\}, \qquad \mcal{R}_{-}(x;P\mcal{K}) =\{(f,g)\in P\mcal{K}\,|\,g(x)=0\}.
\end{gather*}
Then, we can see that $P\mcal{K}_{\{x\},\varnothing}$ is the image of $\mcal{R}_{-}(x;P\mcal{K})$ under the orthogonal projection onto~$\mcal{K}_{\mfrak{X}\backslash\{x\}}$, and similarly, $P\mcal{K}_{\varnothing,\{x\}}$ is the image of $\mcal{R}_{+}(x;P\mcal{K})$ under the orthogonal projection onto~$\mcal{K}_{\mfrak{X}\backslash\{x\}}$.

\subsection[Computation of P(x,0) and P(0,x)]{Computation of $\boldsymbol{P_{\{x\},\varnothing}}$ and $\boldsymbol{P_{\varnothing,\{x\}}}$}

We introduce a lemma found in \cite[Section~7]{BufetovOlshanski2019} that plays prominent roles in computation of~$P_{\{x\},\varnothing}$ and~$P_{\varnothing,\{x\}}$.
\begin{Lemma}[\cite{BufetovOlshanski2019}]\label{lem:projection_reduction_formula}
Let $\mcal{H}$ be a Hilbert space, $\mcal{H}=\mcal{H}_{1}\oplus \mcal{H}_{2}$ be its direct sum decomposition such that $\dim \mcal{H}_{2}<\infty$ and write $\pi_{1}\colon \mcal{H}\to\mcal{H}_{1}$ for the canonical projection. Let $\mcal{L}\subset \mcal{H}$ be a~closed subspace and write the orthogonal projection $P$ onto $\mcal{L}$ in the form of
\begin{equation*}
	P=\left(
	\begin{matrix}
	A & B \\
	C & D
	\end{matrix}
	\right)
\end{equation*}
according to the direct sum decomposition $\mcal{H}=\mcal{H}_{1}\oplus \mcal{H}_{2}$, where $A\in \mfrak{B}(\mcal{H}_{1})$, $B\in \mfrak{B}(\mcal{H}_{2},\mcal{H}_{1})$, $C\in \mfrak{B}(\mcal{H}_{1},\mcal{H}_{2})$, and $D\in \mfrak{B}(\mcal{H}_{2})$.
\begin{enumerate}\itemsep=0pt
\item[$1.$] 	We define a subspace $\mcal{L}_{1}:=\mcal{L}\cap \mcal{H}_{1}$ of $\mcal{H}_{1}$ and write $P_{1}$ for the orthogonal projection onto~$\mcal{L}_{1}$ in~$\mcal{H}_{1}$. If~$D$ is invertible, i.e., $\mcal{L}^{\perp}\cap\mcal{H}_{2}=\{0\}$, we have $P_{1}=A-BD^{-1}C$.
\item[$2.$] 	We define a subspace $\widetilde{\mcal{L}}_{1}:=\pi_{1}(\mcal{L})$ of $\mcal{H}_{1}$ and write $\widetilde{P}_{1}$ for the orthogonal projection onto~$\widetilde{\mcal{L}}_{1}$ in $\mcal{H}_{1}$. If $1-D$ is invertible, i.e., $\mcal{L}\cap\mcal{H}_{2}=\{0\}$, we have $\widetilde{P}_{1}=A+B(1-D)^{-1}C$.
\end{enumerate}
\end{Lemma}
\begin{proof}
Since $P$ is an orthogonal projection, we have $A^{\ast}=A$, $B^{\ast}=C$, $D^{\ast}=D$ and
\begin{gather}\label{eq:relation_component_peojection}
	A^{2}+BC =A, \qquad AB+BD =B, \qquad CA+DC = C, \qquad CB+D^{2} =D.
\end{gather}

1.~Set $\pr{P}_{1}:=A-BD^{-1}C$, which is obviously self-adjoint and is also checked by means of~(\ref{eq:relation_component_peojection}) to be an idempotent.
		Therefore, $\pr{P}_{1}$ is the orthogonal projection onto a closed subspace $\pr{\mcal{L}}_{1}\subset\mcal{H}_{1}$.
		Suppose that $\xi\in\mcal{L}_{1}$. Then $(\xi,0)\in\mcal{H}_{1}\oplus \mcal{H}_{2}$ belongs to $\mcal{L}$ and fixed by $P$:
		\begin{equation*}
		\left(
		\begin{matrix}
		A & B \\
		C & D
		\end{matrix}
		\right)\left(
		\begin{matrix}
		\xi \\
		0
		\end{matrix}
		\right)=\left(
		\begin{matrix}
		\xi \\
		0
		\end{matrix}
		\right),
		\end{equation*}
		which is equivalent to $A\xi=\xi$ and $C\xi=0$.
		Therefore, $\xi$ is fixed by $\pr{P}_{1}$ implying that $\mcal{L}_{1}\subset \pr{\mcal{L}}_{1}$.
		Conversely, suppose that $\xi\in \pr{\mcal{L}}_{1}$. Then, we have $\big(A-BD^{-1}C\big)\xi=\xi$ implying that
		\begin{equation*}
			(\xi,A\xi)-\big(\xi,BD^{-1}C\xi\big)=(\xi,\xi).
		\end{equation*}
		Since $P$ is an orthogonal projection, we have $(\xi,A\xi)\le (\xi,\xi)$. Due to the property $B^{\ast}=C$, the operator $BD^{-1}C$ is a positive operator and $\big(\xi,BD^{-1}C\xi\big)\ge 0$. Therefore, we must have $A\xi=\xi$ and $C\xi=0$ implying $\pr{\mcal{L}}_{1}\subset \mcal{L}_{1}$.

2.~The operator $A+B(1-D)^{-1}C$ is self-adjoint and shown to be an idempotent.
		An element $\xi=(\xi_{1},\xi_{2})\in\mcal{H}_{1}\oplus\mcal{H}_{2}$ belongs to $\mcal{L}$ if and only if
		\begin{equation*}
			A\xi_{1}+B\xi_{2}=\xi_{1},\qquad C\xi_{1}+D\xi_{2}=\xi_{2}.
		\end{equation*}
		Since we have assumed that $1-D$ is invertible, the latter equation gives us $\xi_{2}=(1-D)^{-1}C\xi_{1}$, substitution of which into the former one gives $\big(A+B(1-D)^{-1}C\big)\xi_{1}=\xi_{1}$. Therefore, we obtain the desired result.
\end{proof}

Now, we can obtain explicit expressions of $P_{\{x\},\varnothing}$ and $P_{\varnothing,\{x\}}$.
Let us first write $P\in \mrm{Gr}(\mcal{K},\Gamma)$ as
\begin{equation*}
	P=\left(
	\begin{matrix}
	P_{11} & P_{12} \\
	P_{21} & P_{22}
	\end{matrix}
	\right)
\end{equation*}
according to the direct sum decomposition $\mcal{K}=\ell^{2}(\mfrak{X})\oplus \ell^{2}(\mfrak{X})$, and further write
\begin{equation*}
	P_{ij}=\left(
	\begin{matrix}
		a_{ij} & b_{ij} \\
		c_{ij} & d_{ij}
	\end{matrix}
	\right),\qquad i,j=1,2,
\end{equation*}
according to the direct sum decomposition $\ell^{2}(\mfrak{X})=\ell^{2}(\mfrak{X}\backslash\{x\})\oplus \ell^{2}(\{x\})$.

\begin{Proposition}\label{prop:expression_cond_proj}
Under the above notation, we have, if $P$ is $(\{x\},\varnothing)$-regular,
\begin{equation*}
	P_{\{x\},\varnothing}=\left(
	\begin{matrix}
		a_{11}-b_{12}d_{22}^{-1}c_{21}+b_{11}(1-d_{11})^{-1}c_{11} & a_{12}-b_{12}d_{22}^{-1}c_{22}+b_{11}(1-d_{11})^{-1}c_{12}\\
		a_{21}-b_{22}d_{22}^{-1}c_{21}+b_{21}(1-d_{11})^{-1}c_{11} & a_{22}-b_{22}d_{22}^{-1}c_{22}+b_{21}(1-d_{11})^{-1}c_{12}
	\end{matrix}
	\right)
\end{equation*}
and, if $P$ is $(\varnothing,\{x\})$-regular,
\begin{equation*}
	P_{\varnothing,\{x\}}=\left(
	\begin{matrix}
		a_{11}-b_{11}d_{11}^{-1}c_{11}+b_{12}(1-d_{22})^{-1}c_{21} & a_{12}-b_{11}d_{11}^{-1}c_{12}+b_{12}(1-d_{22})^{-1}c_{22}\\
		a_{21}-b_{21}d_{11}^{-1}c_{11}+b_{22}(1-d_{22})^{-1}c_{21} & a_{22}-b_{21}d_{11}^{-1}c_{12}+b_{22}(1-d_{22})^{-1}c_{22}
	\end{matrix}
	\right)
\end{equation*}
according to the direct sum decomposition $\mcal{K}_{\mfrak{X}\backslash\{x\}}=\mcal{K}^{+}_{\mfrak{X}\backslash\{x\}}\oplus \mcal{K}^{-}_{\mfrak{X}\backslash\{x\}}$.
\end{Proposition}
\begin{proof}
As we saw in Remark \ref{rem:check_anti-sym_matrix_func}, $P_{12}$ and $P_{21}$ are anti-symmetric; in particular, $d_{12}=d_{21}=0$. We also saw that $P_{11}$ and $P_{22}$ are self-adjoint and $JP_{11}J=1-P_{22}$, which implies $d_{11}=1-d_{22}$.

Let us derive the expression of $P_{\{x\},\varnothing}$. By assumption that $P$ is $(\{x\},\varnothing)$-regular, we have $d_{11}\neq 1$ and $d_{22}\neq 0$. Therefore, the desired expression makes sense. It is convenient to write $P$ as
\begin{equation*}
	P=\left(
	\begin{matrix}
	A & b_{1} & b_{2} \\
	c_{1}^{\mrm{T}} & d_{11} & 0 \\
	c_{2}^{\mrm{T}} & 0 & d_{22}
	\end{matrix}
	\right)
\end{equation*}
according to the direct sum decomposition $\mcal{K}=\mcal{K}_{\mfrak{X}\backslash \{x\}}\oplus \mcal{K}_{\{x\}}^{+}\oplus \mcal{K}_{\{x\}}^{-}$, where
\begin{gather*}
	A =\left(
	\begin{matrix}
		a_{11} & a_{12} \\
		a_{21} & a_{22}
	\end{matrix}
	\right), \qquad
	b_{1} =\left(
	\begin{matrix}
		b_{11} \\
		b_{21}
	\end{matrix}
	\right), \qquad
	b_{2} =\left(
	\begin{matrix}
		b_{12} \\
		b_{22}
	\end{matrix}
	\right), \\
	c_{1}^{\mrm{T}} =\left(
	\begin{matrix}
		c_{11} & c_{12}
	\end{matrix}
	\right), \qquad
	c_{2}^{\mrm{T}} =\left(
	\begin{matrix}
		c_{21} & c_{22}
	\end{matrix}
	\right).
\end{gather*}
Notice that $\mcal{R}_{-}(x;\mcal{K})=\mcal{K}_{\mfrak{X}\backslash\{x\}}\oplus \mcal{K}_{\{x\}}^{+}$.
Since $d_{22}\neq 0$, we can apply the formula (1) in Lemma~\ref{lem:projection_reduction_formula} for $\mcal{H}_{1}=\mcal{R}_{-}(x;\mcal{K})$ and $\mcal{H}_{2}=\mcal{K}_{\{x\}}^{-}$ to express the orthogonal projection onto $\mcal{R}_{-}(x;P\mcal{K})$ in $\mcal{R}_{-}(x;\mcal{K})$ as
\begin{equation*}
	\left(
	\begin{matrix}
	A & b_{1} \\
	c_{1}^{\mrm{T}} & d_{11}
	\end{matrix}
	\right)-
	\left(
	\begin{matrix}
		b_{2} \\
		0
	\end{matrix}
	\right)d_{22}^{-1}\left(
	\begin{matrix}
		c_{2}^{\mrm{T}} & 0
	\end{matrix}
	\right)=
	\left(
	\begin{matrix}
		A-b_{2}d_{22}^{-1}c_{2}^{\mrm{T}} & b_{1} \\
		c_{1}^{\mrm{T}} & d_{11}
	\end{matrix}
	\right)
\end{equation*}
according to the direct sum decomposition $\mcal{R}_{-}(x;\mcal{K})=\mcal{K}_{\mfrak{X}\backslash\{x\}}\oplus \mcal{K}_{\{x\}}^{+}$.
Recall that $P\mcal{K}_{\{x\},\varnothing}$ is the image of $\mcal{R}_{-}(x;P\mcal{K})$ under the projection onto $\mcal{K}_{\mfrak{X}\backslash\{x\}}$.
Now, since $d_{11}\neq 1$, we can apply the formula~(2) of Lemma~\ref{lem:projection_reduction_formula} for $\mcal{H}_{1}=\mcal{K}_{\mfrak{X}\backslash\{x\}}$ and $\mcal{H}_{2}=\mcal{K}_{\{x\}}^{+}$ to express $P_{\{x\},\varnothing}$ as
\begin{equation*}
	P_{\{x\},\varnothing}=A-b_{2}d_{22}^{-1}c_{2}^{\mrm{T}}+b_{1}(1-d_{11})^{-1}c_{1}^{\mrm{T}},
\end{equation*}
which is exactly the desired result for $P_{\{x\},\varnothing}$.
The expression of $P_{\varnothing,\{x\}}$ can also be derived in a~similar way to prove the desired result.
\end{proof}

\subsection{Correlation functions of conditional measures}
Let us proceed to computation of the correlation functions under conditional measures.
Let us take $P\in \mrm{Gr}(\mcal{K},\Gamma)$ that is $(\{x\},\varnothing)$-regular. In this case $M^{P}(C(\{x\},\varnothing))=\mbb{K}^{M^{P}}_{12}(x,x)=(e_{x},P_{22}e_{x})$ is non-zero. Therefore, the conditional measure $M^{P}_{\{x\},\varnothing}$ on $\Omega (\mfrak{X}\backslash\{x\})$ is well-defined. Due to the arguments in Section~\ref{subsect:conditioning_CAR_alg}, a correlation function of $M^{P}_{\{x\},\varnothing}$ is computed as
\begin{gather*}
	\rho^{M^{P}_{\{x\},\varnothing}}(x_{1},\dots, x_{n}) =\frac{\varphi_{P}(a_{x_{1}}^{\ast}\cdots a_{x_{n}}^{\ast}a_{x_{n}}\cdots a_{x_{1}}a_{x}^{\ast}a_{x})}{\varphi_{P}(a_{x}^{\ast}a_{x})}
	 =\frac{\pf \left[\mbb{K}_{P}(x_{i},x_{j})\right]_{0\le i,j\le n}}{\mbb{K}^{M^{P}}_{12}(x,x)},
\end{gather*}
where we set $x_{0}:=x$ and $x_{1},\dots, x_{n}\in \mfrak{X}\backslash \{x\}$ are distinct.
It is standard to express the above quantity in terms of the Pfaffian of a matrix of smaller size.
In fact, by means of the formula $\pf (B^{\mrm{T}}AB)=(\det B)(\pf A)$ for an anti-symmetric matrix $A$ and a matrix $B$, we can see that
\begin{equation*}
	\rho^{M^{P}_{\{x\},\varnothing}}(x_{1},\dots, x_{n})=\pf \big[\widetilde{\mbb{K}}(x_{i},x_{j})\big]_{1\le i,j\le n},
\end{equation*}
where
\begin{gather*}
\widetilde{\mbb{K}}_{11}(y,z)=\mbb{K}^{M^{P}}_{11}(y,z)-\frac{\mbb{K}^{M^{P}}_{12}(y,x)\mbb{K}^{M^{P}}_{11}(x,z)}{\mbb{K}^{M^{P}}_{12}(x,x)}+\frac{\mbb{K}^{M^{P}}_{11}(y,x)\mbb{K}^{M^{P}}_{12}(x,z)}{\mbb{K}^{M^{P}}_{12}(x,x)}, \\
\widetilde{\mbb{K}}_{12}(y,z)=\mbb{K}^{M^{P}}_{12}(y,z)-\frac{\mbb{K}^{M^{P}}_{12}(y,x)\mbb{K}^{M^{P}}_{12}(x,z)}{\mbb{K}^{M^{P}}_{12}(x,x)}+\frac{\mbb{K}^{M^{P}}_{11}(y,x)\mbb{K}^{M^{P}}_{22}(x,z)}{\mbb{K}^{M^{P}}_{12}(x,x)}, \\
\widetilde{\mbb{K}}_{22}(y,z)=\mbb{K}^{M^{P}}_{22}(y,z)-\frac{\mbb{K}^{M^{P}}_{22}(y,x)\mbb{K}^{M^{P}}_{12}(x,z)}{\mbb{K}^{M^{P}}_{12}(x,x)}+\frac{\mbb{K}^{M^{P}}_{12}(y,x)\mbb{K}^{M^{P}}_{22}(x,z)}{\mbb{K}^{M^{P}}_{12}(x,x)},
\end{gather*}
for $y,z\in\mfrak{X}\backslash\{x\}$. Note that the other component $\widetilde{\mbb{K}}_{21}(y,z)$ is determined by the anti-symmetry $\widetilde{\mbb{K}}(y,z)^{\mrm{T}}=-\widetilde{\mbb{K}}(z,y)$.

We next assume that $P\in \mrm{Gr}(\mcal{K},\Gamma)$ is $(\varnothing,\{x\})$-regular. Then, $M^{P}(C(\varnothing,\{x\}))=1-\mbb{K}^{M^{P}}_{12}(x,x)=(e_{x},P_{11}e_{y})$ is non-zero. Therefore, the conditional measure $M^{P}_{\varnothing,\{x\}}$ on $\Omega (\mfrak{X}\backslash\{x\})$ is well-defined. A correlation function is computed as
\begin{gather*}
	\rho^{M^{P}_{\varnothing,\{x\}}}(x_{1},\dots, x_{n}) =\frac{\varphi_{P}(a_{x_{1}}^{\ast}\cdots a_{x_{n}}^{\ast}a_{x_{n}}\cdots a_{x_{1}}a_{x}a^{\ast}_{x})}{\varphi_{P}(a_{x}a^{\ast}_{x})}
	 =\frac{\pf [\pr{\mbb{K}}_{P}(x_{i},x_{j}) ]_{0\le i,j\le n}}{1-\mbb{K}^{M^{P}}_{12}(x,x)},
\end{gather*}
where we set $x_{0}:=x$, $x_{1},\dots, x_{n}\in \mfrak{X}\backslash\{x\}$ are distinct and
\begin{gather*}
	\pr{\mbb{K}}_{P}(x_{i},x_{j}):=
	\begin{cases}
		\mbb{K}_{P}(x_{i},x_{j}), & 1\le i,j\le n, \\
		\left(
		\begin{matrix}
			(e_{x},P_{11}e_{x_{j}}) & (e_{x},P_{12}e_{x_{j}}) \\
			(e_{x}, P_{21}e_{x_{j}}) & (e_{x},P_{22}e_{x_{j}})
		\end{matrix}
		\right), & i=0,\quad 1\le j\le n, \vspace{1mm}\\
		\left(
		\begin{matrix}
			(e_{x_{i}},P_{22}e_{x}) & (e_{x_{i}},P_{21}e_{x}) \\
			(e_{x_{i}}, P_{12}e_{x}) & (e_{x_{i}},P_{11}e_{x})
		\end{matrix}
		\right), & 1\le i \le n,\quad j=0, \vspace{1mm}\\
		\left(
		\begin{matrix}
			0 & (e_{x},P_{11}e_{x}) \\
			(e_{x}, (P_{22}-1)e_{x}) & 0
		\end{matrix}
		\right), & i=j=0.
	\end{cases}
\end{gather*}
Again, we write the correlation function in terms of the Pfaffian of a $(2n\times 2n)$-anti-symmetric matrix. Consequently, we obtain
\begin{equation*}
	\rho^{M^{P}_{\varnothing,\{x\}}}(x_{1},\dots, x_{n})=\pf\big[\what{\mbb{K}}(x_{i},x_{j})\big]_{1\le i,j\le n},
\end{equation*}
where
\begin{gather*}
\what{\mbb{K}}_{11}(y,z) =\mbb{K}^{M^{P}}_{11}(y,z)-\frac{\mbb{K}^{M^{P}}_{11}(y,x)\mbb{K}^{M^{P}}_{12}(x,z)}{1-\mbb{K}^{M^{P}}_{12}(x,x)}+\frac{\mbb{K}^{M^{P}}_{12}(y,x)\mbb{K}^{M^{P}}_{11}(x,z)}{1-\mbb{K}^{M^{P}}_{12}(x,x)}, \\
\what{\mbb{K}}_{12}(y,z) =\mbb{K}^{M^{P}}_{12}(y,z)-\frac{\mbb{K}^{M^{P}}_{11}(y,x)\mbb{K}^{M^{P}}_{22}(x,z)}{1-\mbb{K}^{M^{P}}_{12}(x,x)}+\frac{\mbb{K}^{M^{P}}_{12}(y,x)\mbb{K}^{M^{P}}_{12}(x,z)}{1-\mbb{K}^{M^{P}}_{12}(x,x)}, \\
\what{\mbb{K}}_{22}(y,z) =\mbb{K}^{M^{P}}_{22}(y,z)-\frac{\mbb{K}^{M^{P}}_{12}(y,x)\mbb{K}^{M^{P}}_{22}(x,z)}{1-\mbb{K}^{M^{P}}_{12}(x,x)}+\frac{\mbb{K}^{M^{P}}_{22}(y,x)\mbb{K}^{M^{P}}_{12}(x,z)}{1-\mbb{K}^{M^{P}}_{12}(x,x)}
\end{gather*}
for $y,z\in\mfrak{X}\backslash\{x\}$, and the other component is determined by the anti-symmetry $\what{\mbb{K}}(y,z)^{\mrm{T}}=-\what{\mbb{K}}(z,y)$.

\subsection{Proof of Theorem \ref{thm:conditional_projection}}
It remains to show that $\widetilde{\mbb{K}}(y,z)=\mbb{K}_{P_{\{x\},\varnothing}}(y,z)$ and $\what{\mbb{K}}(y,z)=\mbb{K}_{P_{\varnothing,\{x\}}}(y,z)$, $y,z\in \mfrak{X}\backslash\{x\}$.
It is straightforward that
\begin{gather*}
\widetilde{\mbb{K}}_{11}(y,z)=(e_{y},P_{21}e_{z})-\frac{(e_{y},P_{22}e_{x})(e_{x},P_{21}e_{z})}{(e_{x},P_{22}e_{x})}+\frac{(e_{y},P_{21}e_{x})(e_{x},P_{11}e_{z})}{(e_{x},(1-P_{11})e_{x})}, \\
\widetilde{\mbb{K}}_{12}(y,z)=(e_{y},P_{22}e_{z})-\frac{(e_{y},P_{22}e_{x})(e_{x},P_{22}e_{z})}{(e_{x},P_{22}e_{x})}+\frac{(e_{y},P_{21}e_{x})(e_{x},P_{12}e_{z})}{(e_{x},(1-P_{11})e_{x})}, \\
\widetilde{\mbb{K}}_{22}(y,z)=(e_{y},P_{12}e_{z})-\frac{(e_{y},P_{12}e_{x})(e_{x},P_{22}e_{z})}{(e_{x},P_{22}e_{x})}+\frac{(e_{y},P_{11}e_{x})(e_{x},P_{12}e_{z})}{(e_{x},(1-P_{11})e_{x})}.
\end{gather*}
Comparing these descriptions with the result of Proposition~\ref{prop:expression_cond_proj}, we conclude that $\widetilde{\mbb{K}}(y,z)=\mbb{K}_{P_{\{x\},\varnothing}}(y,z)$, $y,z\in\mfrak{X}\backslash\{x\}$.
It also follows from
\begin{gather*}
\what{\mbb{K}}_{11}(y,z)=(e_{y},P_{21}e_{z})-\frac{(e_{y}P_{21}e_{x})(e_{x},P_{11}e_{z})}{(e_{x},P_{11}e_{x})}+\frac{(e_{y},P_{22}e_{x})(e_{x},P_{21}e_{z})}{(e_{x},(1-P_{22})e_{x})}, \\
\what{\mbb{K}}_{12}(y,z)=(e_{y},P_{22}e_{z})-\frac{(e_{y}P_{21}e_{x})(e_{x},P_{12}e_{z})}{(e_{x},P_{11}e_{x})}+\frac{(e_{y},P_{22}e_{x})(e_{x},P_{22}e_{z})}{(e_{x},(1-P_{22})e_{x})}, \\
\what{\mbb{K}}_{22}(y,z)=(e_{y},P_{12}e_{z})-\frac{(e_{y}P_{11}e_{x})(e_{x},P_{12}e_{z})}{(e_{x},P_{11}e_{x})}+\frac{(e_{y},P_{12}e_{x})(e_{x},P_{22}e_{z})}{(e_{x},(1-P_{22})e_{x})}
\end{gather*}
and Proposition \ref{prop:expression_cond_proj} that the other desired coincidence $\what{\mbb{K}}(y,z)=\mbb{K}_{P_{\varnothing,\{x\}}}(y,z)$, $y,z\in \mfrak{X}\backslash\{x\}$ holds.

\appendix

\section{Shifted Schur measures}\label{sect:shifed_Schur_measures}
The shifted Schur measures were introduced in \cite{TracyWidom2004} associated with the Schur $Q$-functions. It follows from the result in \cite{Matsumoto2005} that they are defined in terms of quasi-free states of a CAR algebra, as we overview in this appendix. Note that a free fermionic approach to shifted Schur measures has also been proposed in~\cite{WangLi2019} relying on a different algebra from the one adopted in~\cite{Matsumoto2005} and here.

\subsection{Definition of shifted Schur measures}
The notations regarding symmetric functions are inherited from Sestion~\ref{subsect:Schur_measures}.
A partition $\lambda=(\lambda_{1}\ge\lambda_{2}\ge\cdots)\in\mbb{Y}$ is said to be strict if $\lambda_{1}>\lambda_{2}>\cdots$.
We write $\mbb{D}$ for the collection of strict partitions. We do not contain here a definition of the Schur $Q$-functions (see \cite[Chapter~III, Section~8]{Macdonald1999}), but just say that, for each strict partition $\lambda\in \mbb{D}$, the Schur $Q$-function $Q_{\lambda}\in\Lambda$ is defined as the Macdonald symmetric function at the parameter $(q,t)=(0,-1)$ and the Schur $P$-function is $P_{\lambda}=2^{-\ell(\lambda)}Q_{\lambda}$, where $\ell(\lambda)$ is the length of the partition. Another significant property of the Schur $Q$-functions is that, when we write $\Lambda^{\mrm{odd}}$ for the subring generated by power-sum symmetric functions $p_{n}$, $n=1,3,5,\dots$ of odd degree, then $Q_{\lambda}$, $\lambda\in\mbb{D}$ form a basis of $\Lambda^{\mrm{odd}}$.

Positive specializations of the Schur $Q$-functions were classified in \cite{Ivanov1999, Nazarov1988}. Let $\mbb{T}_{Q}$ be the collection of data $\rho=(\alpha_{1}\ge \alpha_{2}\ge\cdots \ge 0)$ such that $\sum_{j\ge 1}\alpha_{j}\le 1$. Associated to such data~$\rho$, we define an algebraic homomorphism $\tau_{\rho}\colon \Lambda^{\mrm{odd}}\to\mbb{C}$ by $\tau_{\rho}(p_{1})=1$ and $\tau_{\rho}(p_{n})=\sum_{g\ge 1}\alpha_{j\ge 1}^{n}$ for $n=3,5,\dots$. As usual, for a symmetric function $F\in \Lambda^{\mrm{odd}}$ and $\rho\in \mbb{T}_{Q}$, we write the specialization as $F(\rho):=\tau_{\rho}(F)$. Then, it is known that $Q_{\lambda}(\rho)\ge 0$ for all $\lambda\in\mbb{D}$.

Let $\mfrak{X}=\mbb{N}=\{1,2,\dots\}$. Then, a strict partition $\lambda\in\mbb{D}$ gives a subset $\{\lambda_{i}\}_{i=1}^{\ell(\lambda)}\subset\mfrak{X}$, inducing a bijection $\mbb{D}\to \Omega:=\Omega(\mfrak{X})$. We consider the subset $\mbb{T}_{Q}^{\circ}$ of $\mbb{T}_{Q}$ that consists of data such that $\alpha_{1}<1$. Equivalently, it is $\mbb{T}_{Q}^{\circ}=\mbb{T}_{Q}\backslash\{(1,0,\dots)\}$.
For $\rho\in\mbb{T}_{Q}^{\circ}$, the associated shifted Schur measure $M_{Q(\rho)}$ is the probability measure on $(\Omega,\Sigma)$ defined by
\begin{equation*}
	M_{Q(\rho)}(\lambda)\propto Q_{\lambda}(\rho)P_{\lambda}(\rho),\qquad \lambda\in\mbb{D}\simeq \Omega.
\end{equation*}
Note that, in general, the specializations for $Q_{\lambda}$ and $P_{\lambda}$ can be different. In this sense, we focus on special cases of shifted Schur measures. Some particular examples related to projective characters of the infinite symmetric group have been studied in \cite{Borodin1999,Petrov2010,Petrov2011}.

\subsection{CAR algebra formalism}
Recall that we are working on $\mfrak{X}=\mbb{N}$. Let $P_{0}\in \mrm{Gr}(\mcal{K},\Gamma)$ be the projection defined by (\ref{eq:projection_vacuum}) and consider the corresponding Fock representation $\big(\pi_{P_{0}},\mcal{F}\big(\ell^{2}(\mfrak{X})\big)\big)$.

We write $S(-I)$ for the second quantization of $-I\in\mfrak{U}\big(\ell^{2}(\mfrak{X})\big)$ (see \cite[Section~5.2.1]{BratteliRobinson1997}).
For a positive odd integer $n$, we define
\begin{equation*}
h_{n}:=\sum_{j=1}^{\infty}\pi_{P_{0}}(a_{j}^{\ast}a_{n+j})+\frac{1}{\sqrt{2}}S(-I)\pi_{P_{0}}(a_{n})+\sum_{j=1}^{\frac{n-1}{2}}(-1)^{j}\pi_{P_{0}}(a_{j}a_{n-j})
\end{equation*}
and
\begin{equation*}
h_{-n}:=h_{n}^{\ast}=\sum_{j=1}^{\infty}\pi_{P_{0}}(a^{\ast}_{n+j}a_{j})+\frac{1}{\sqrt{2}}\pi_{P_{0}}(a_{n}^{\ast})S(-I)+\sum_{j=1}^{\frac{n-1}{2}}(-1)^{j}\pi_{P_{0}}(a^{\ast}_{n-j}a^{\ast}_{j}).
\end{equation*}
Then, these operators exhibit the Heisenberg commutation relations $[h_{m},h_{n}]=\frac{m}{2}\delta_{m+n,0}$.

Let $\rho\in\mbb{T}_{Q}^{\circ}$ and set
\begin{equation*}
	\Xi_{\pm}(\rho):=\exp\left(\sum_{n\in 2\mbb{Z}_{\ge 0}+1}\frac{p_{n}(\rho)}{n}h_{\pm n}\right).
\end{equation*}
Then, it is obvious that $\Xi_{\pm}(\rho)^{\ast}=\Xi_{\mp}(\rho)$.
The verification of these operators goes in a similar manner as the one in Section~\ref{subsect:Schur_measures} for Schur measures.
The following theorem is an immediate consequence of~\cite{Matsumoto2005}.

\begin{Theorem}For $\rho\in\mbb{T}_{Q}^{\circ}$, the functional $\varphi_{Q(\rho)}$ defined by
\begin{equation*}
\varphi_{Q(\rho)}(A):=\frac{(\Xi_{-}(\rho)\bm{1},\pi_{P_{0}}(A)\Xi_{-}(\rho)\bm{1})_{\mcal{F}(\ell^{2} (\mfrak{X}))}}{\|\Xi_{-}(\rho)\bm{1}\|^{2}},\qquad A\in \mcal{C}(\mcal{K},\Gamma)
\end{equation*}
is a quasi-free state over $\mcal{C}(\mcal{K},\Gamma)$ and
\begin{equation*}
	\varphi_{Q(\rho)}(f)=\int_{\Omega}f(\omega)M_{Q(\rho)}({\rm d}\omega),\qquad f\in C(\Omega).
\end{equation*}
\end{Theorem}
\begin{proof}[Outline of proof]
A significant observation is as follows:
recall that the Fock space $\mcal{F}\big(\ell^{2}(\mfrak{X})\big)$ admits a complete orthonormal system
\begin{equation*}
	e_{\lambda}=e_{\lambda_{1}}\wedge\cdots \wedge e_{\lambda_{\ell}},\qquad \lambda=(\lambda_{1}>\cdots>\lambda_{\ell})\in\mbb{D}.
\end{equation*}
We have \cite[Proposition~3.2]{Matsumoto2005}
\begin{equation}\label{eq:expansion_vacuum_shifted_Schur}
	\Xi_{-}(\rho)\bm{1}=\sum_{\lambda\in\mbb{D}}2^{-\ell(\lambda)/2}Q_{\lambda}(\rho)e_{\lambda}.
\end{equation}
The remaining part of the proof is similar to that of Proposition~\ref{prop:Schur_measure_quasifree} for the Schur measures.
\end{proof}

From the definition of $M_{Q(\rho)}$ and the expansion~(\ref{eq:expansion_vacuum_shifted_Schur}), we have
\begin{equation*}
	\frac{1}{\|\Xi_{-}(\rho)\bm{1}\|}\Xi_{-}(\rho)\bm{1}=\sum_{\lambda\in\mbb{D}}M_{Q(\rho)}(\lambda)^{1/2}e_{\lambda}.
\end{equation*}
Therefore, relying on a similar argument as the proof of Theorem~\ref{thm:perfectness_Schur}, we have the following result.
\begin{Theorem}
Assume that a shifted Schur measure $M_{Q(\rho)}$ is $\mfrak{S}$-quasi-invariant. Then it is perfect.
\end{Theorem}

\subsection*{Acknowledgements}

The author is grateful to Makoto Katori, Tomoyuki Shirai and Sho Matsumoto for discussions and comments on the manuscript.
He also thanks the anonymous referees for their useful suggestions for improvements of the manuscript.
This work was supported by the Grant-in-Aid for JSPS Fellows (No.~19J01279).

\pdfbookmark[1]{References}{ref}
\LastPageEnding

\end{document}